\documentclass[bj,preprint,a4paper]{imsart}

\usepackage{amsthm,amsmath,amsbsy,amsfonts,amssymb}
\usepackage[numbers]{natbib}
\usepackage{graphicx,verbatim}
\usepackage{booktabs,stmaryrd}

\usepackage{algorithm}
\usepackage{algpseudocode}

\RequirePackage[colorlinks,citecolor=blue,urlcolor=blue]{hyperref}

\usepackage{amsmath,examplep,wasysym}
\usepackage{a4wide}
\usepackage{dsfont}
\usepackage{amssymb , amsmath, amsthm, amsopn, amstext, amscd}
\usepackage{graphicx}
\usepackage{multirow}

\DeclareMathOperator{\sign}{sgn}

\DeclareMathOperator*{\argmin}{arg\,min}

\def\hat{\widehat}
\usepackage[usenames,dvipsnames]{color}
\newcommand{\cnew}[1]{\textcolor{red}{#1}}

\newcommand{\rk}{\operatorname{rank}}

\newcommand{\mr}{\mathbb{R}}

\newcommand{\mn}{\mathbb{N}}

\newcommand{\betaz}{\bbeta^*}
\newcommand{\betazT}{\bbeta^{*,T}}

\newcommand{\betae}{\hatbbeta}
\newcommand{\betaL}{\hat\bbeta{}^{\rm Lasso}_\lambda}

\newcommand{\eps}{\epsilon}
\newcommand{\beps}{\boldsymbol{\epsilon}}

\newcommand{\bbeta}{\boldsymbol{\beta}}
\newcommand{\bzeta}{\boldsymbol{\zeta}}
\newcommand{\barbbeta}{\;\bar{\!\!\boldsymbol{\beta}}}
\newcommand{\hatbbeta}{\,\hat{\!\boldsymbol{\beta}}}

\newcommand{\bdelta}{\boldsymbol{\delta}}
\newcommand{\bomega}{\boldsymbol{\omega}}
\newcommand{\bxi}{\boldsymbol{\xi}}

\newcommand{\bfX}{\mathbf X}
\newcommand{\bfA}{\mathbf A}
\newcommand{\bfE}{\mathbf E}

\newcommand{\bfP}{\mathbf P}
\newcommand{\bfI}{\mathbf I}
\newcommand{\bfU}{\mathbf U}
\newcommand{\bfV}{\mathbf V}
\newcommand{\bLambda}{\boldsymbol\Lambda}

\newcommand{\ba}{\mathbf a}
\newcommand{\be}{\boldsymbol e}
\newcommand{\bs}{\boldsymbol s}
\newcommand{\bt}{\boldsymbol t}
\newcommand{\bu}{\boldsymbol u}
\newcommand{\bv}{\boldsymbol v}
\newcommand{\boldf}{\boldsymbol f}
\newcommand{\boldg}{\boldsymbol g}

\newcommand{\bfs}{\mathbf s}

\newcommand{\bx}{\boldsymbol x}

\newcommand{\by}{\boldsymbol y}

\newtheorem{theorem}{Theorem}[]
\newtheorem{proposition}{Proposition}[]

\newtheorem{corollary}{Corollary}[]

\theoremstyle{remark}
\newtheorem{remark}{Remark}[]
\newtheorem{example}{Example}[]
\newtheorem{definition}{Definition}[]

\volume{23}
\issue{1}
\pubyear{2017}
\firstpage{552}
\lastpage{581}
\doi{10.3150/15-BEJ756}% Updated by VTEXPTS2LaTeX.exe, 12.11.2015 14:56

\usepackage{pdfsync}

\usepackage{tikz}

\begin{document}

\begin{frontmatter}
\title{On the prediction performance of the Lasso}
\runtitle{Refined analysis of the Lasso}

\begin{aug}
\author{\fnms{Arnak S.} \snm{Dalalyan}\thanksref{a},\ead[label=e1]{arnak.dalalyan@ensae.fr}
\ead[label=u1,url]{http://arnak-dalalyan.fr/}}
\author{\fnms{Mohamed} \snm{Hebiri}\thanksref{b},\ead[label=e2]{mohamed.hebiri@univ-mlv.fr}
\ead[label=u2,url]{http://perso-math.univ-mlv.fr/users/hebiri.mohamed/}}
\and
\author{\fnms{Johannes} \snm{Lederer}\thanksref{c}\ead[label=e3]{johanneslederer@cornell.edu}
\ead[label=u3,url]{http://www.johanneslederer.de/}}

\address[a]{ENSAE-CREST, 3 Avenue Pierre Larousse, 92240 Malakoff, France.
\printead{e1}}

\address[b]{Universit\'e Paris-Est -- Marne-la-Vall\'ee, 5 boulevard Descartes, Champs sur Marne, \\ 77454 Marne-la-Vall\'ee, Cedex 2 France.
\printead{e2}}

\address[c]{Cornell University, 1188 Comstock Hall, Ithaca, NY 14853-2601.
\printead{e3}}

%\thankstext{T1}{Research partially supported by ANR Parcimonie}
%\thankstext{t2}{First supporter of the project}
%\thankstext{t3}{Second supporter of the project}
\runauthor{Dalalyan, Hebiri, and Lederer}

\affiliation{ENSAE-CREST, Universit\'e Paris Est, and Cornell University}

\end{aug}

% HISTORY:
\received{\smonth{10} \syear{2014}}% Updated by VTEXPTS2LaTeX.exe,
%12.11.2015 14:56
%
\revised{\smonth{7} \syear{2015}}% Updated by VTEXPTS2LaTeX.exe,
%12.11.2015 14:56

\begin{abstract}
Although the Lasso has been extensively studied, the relationship between its prediction performance
and the correlations of the covariates is not fully understood. In this paper, we give new
insights into this relationship in the context of multiple linear regression. We show, in particular, that the incorporation of a simple
correlation measure into the tuning parameter can lead to a nearly optimal prediction performance of the Lasso even for highly
correlated covariates. However, we also reveal that for moderately correlated covariates, the prediction performance of the
Lasso can be mediocre irrespective of the choice of the tuning parameter. We finally show that our results also lead to near-optimal
rates for the least-squares estimator with total variation penalty.
\end{abstract}



\end{frontmatter}

\date{}

\maketitle

\section{Introduction}
In recent years, considerable effort has been devoted to establishing sharp theoretical guarantees for the prediction performance\footnote{Following \cite{BRT}, the term prediction performance is understood here as the magnitude of the risk measured by the
prediction loss $\frac1n\|\bfX(\betae-\betaz)\|_2^2$. This is not to be confused with the understanding of this term
 more common in machine learning literature, where the aim is to predict the label of a new unlabeled feature.}  of the Lasso \cite{Tibshirani-LASSO}.
Although there are already risk bounds for a variety of settings~\cite{Koltchinskii_book,Buhlmann11},
the prediction performance of the Lasso is still not completely understood. In this paper, we review and improve the sharpest known risk
bounds to gain new insight into the prediction performance of the Lasso.

Our approach is valid for a broad class of models, but to avoid digression, we study the prediction performance of the Lasso only for  Gaussian linear regression models with deterministic design. More specifically, we consider data consisting of $n$ random observations $y_1,\dots, y_n\in\mr$ and $p$ fixed covariates $\bx^1,\ldots,\bx^p\in\mr^{n}$. We further assume that there is a regression vector $\bbeta^*\in\mr^p$ and a noise level $\sigma^*>0$ such that the residuals $y_i-\beta_1^*(\bx^1)_i-\ldots-\beta^*_p(\bx^p)_i$ are identically and independently distributed according to  a centered Gaussian distribution with variance $\sigma^*{}^2$. In vector notation, this reads
\begin{equation}\label{model}
\by = \bfX\bbeta^*+\bxi,\qquad \bxi\sim \sigma^*\mathcal N_n(0,\bfI_n),
\end{equation}
where $\by:=(y_1,\dots, y_n)^\top\in\mr^n$ is the response vector, $\bfX:=(\bx^1,\ldots,\bx^p)\in\mr^{n\times p}$ the design matrix (for which we assume, without loss of generality, that $\|\bx^j\|_2^2\le n$ for all $j\in\{1,\dots,p\}$), $\xi\in\mr^n$ the noise vector, and $\bfI_n$ denotes the identity matrix in $\mr^{n\times n}$. To keep the exposition simple, we restrict ourselves to Gaussian distributions for the noise vector and to fixed covariates, but our results extend to more general classes of distributions and, if the results are understood conditionally on the covariates, also hold for random covariates. Next, we recall that the Lasso is any solution of the convex optimization problem
\begin{equation}\label{lasso}
\betaL \in \text{arg}\min_{\bbeta} \Big\{\frac{1}{2n}\|\by-\bfX\bbeta\|_2^2+\lambda\|\bbeta\|_1\Big\},
\end{equation}
that can be efficiently solved even for very large values of $p$ and $n$ \cite{Efron-LARS,BachJMO12}.
The magnitude of the tuning parameter $\lambda>0$ determines the amount of penalization and, therefore, has a crucial influence on the performance of the Lasso.
Note that, in particular for high dimensional models where $p>n$, there are typically multiple solutions of~\eqref{lasso}; however, since $\bfX\betae_\lambda=\bfX\betae'_\lambda$ for any two solution~$\betae_\lambda,~\betae'_\lambda$, all solutions have the same prediction performance.

In this paper, we study the prediction performance of the Lasso and make, in particular, the following five contributions:
\begin{enumerate}

\item Numerous empirical results indicate that the prediction error of the Lasso with the
universal tuning parameter $\lambda = \sqrt{2\log(p)/n}$ is at most proportional to $\frac{\log(p)}{n}\times\rk(\bfX)$. Equation~\eqref{eq:2.51} of this paper is the first theoretical confirmation of this conjecture.

\item
For sparse vectors $\betaz$ with support $J^*=\{j\in\{1,\ldots,p\}:\beta^*_j\not=0\}$ and for covariates that are strongly correlated in the sense that all irrelevant covariates $\{\bx^j:j\not\in J^*\}$ are close to the linear span of relevant covariates $\{\bx^j:j\in J^*\}$, empirical results suggest that
the smallest prediction loss is obtained choosing a tuning parameter $\lambda$ that is substantially smaller than the universal one. The influence of correlations on the prediction performance of the Lasso was first considered in~\cite{vdGeer11,HebiriLederer}, where tuning parameters smaller than the classical ones are suggested if the covariates are correlated. In particular, for rates of convergence substantially faster than the slow rate $(s^*/n)^{1/2}$ (where $s^*=|J^*|$), their results suggest to incorporate the geometry of the covariates via a function of the entropy numbers of the symmetric convex hull of the covariates into the tuning parameters. One contribution of this paper is to advance and complement these results: First, we introduce a new measure for the geometry of the covariates. In contrast to earlier measures, the new measure is computable and allows for an exhaustive characterization of Lasso prediction as a function of the correlations. Second, Corollary~\ref{cor:slow} establishes that even with the universal tuning parameter (which does not incorporate the geometry of the covariates), Lasso can have fast rates of convergence in highly correlated settings.

These results relate Lasso prediction with the choice of the tuning parameter and, in particular, provide a description of the Lasso prediction performance for optimal tuning parameters. We expect, therefore, that our results could be complemented by the current research on tuning parameter calibration~\cite{Belloni2011,Lederer14a}.

\item For really sparse vectors, that is, for $s^*$ considerably smaller than $n$ (for example, $s^*$ is fixed and $n\to\infty$),
there are methods that satisfy fast rate bounds for prediction irrespective of the correlations of the
covariates~\cite{BTWAggSOI,DT07,RigTsy11,DT12a,DT12b}. Fast rate bounds for Lasso prediction, in contrast, usually rely on
assumptions on the correlations of the covariates such as low coherence \cite{CandesPlan09}, restricted eigenvalues \cite{BRT,Raskutti10},
restricted isometry \cite{CandesTao07}, compatibility \cite{vandeGeer07}, cone invertibilty~\cite{Ye10}, \textit{etc.} For Lasso prediction,
it is therefore not known whether fast rate bounds are available irrespective of the correlations of the covariates. This question is open
even if we allow for oracle choices of the tuning parameter $\lambda$, that is, if we allow for $\lambda$ that depend on the true regression
vector $\betaz$, the noise vector $\bxi$, and the noise level $\sigma^*$. In the present work, we give a negative respose to this question in  Exemple~\ref{example1}.

\item Known results imply fast rates for prediction with the Lasso in the following two extreme cases:
First, when the covariates are mutually orthogonal, and second, when the covariates are all collinear. But how far from these two extreme cases can a design be such that it still permits fast rates for prediction with the Lasso? For the first case, the case of mutually orthogonal covariates, this question has been thoroughly studied~\cite{BRT,BTW07,Zhang09,VandeGeerConditionLasso09,Wainwright09,Cai2010,JudNem11b}. For the second case, the case of collinear covariates, this question has received much less attention. Therefore, this question is one of our main topics, and we give answers in Corollary~\ref{cor:slow} and Proposition~\ref{prop:4.1}.

\item We finally show that our new Lasso prediction guarantees also lead to optimal guarantees for prediction with total variation penalties. Total variation penalties can enforce similarities between neighboring pixels or between values of signals and are therefore popular for image denoising and signal processing. However, the known theoretical results for prediction with total variation penalties are fragmentary. Using completely new probabilistic approaches, we relate prediction with total variation penalties with our prediction bounds for the Lasso. This allows us to state in Propositions~\ref{prop:3.2}, \ref{prop:4.2}, and~\ref{prop:4.3} and in Equation~\eqref{eq:riskTV1}  near-optimal guarantees for prediction with total variation penalties for a large variety of settings.

\end{enumerate}

Let us stress that we focus only on the behavior of the Lasso in terms of the prediction loss and do not explore here such important
aspects of the Lasso as variable selection and estimation.
Moreover, while we give some insights into computational aspects, this work is essentially a theoretical contribution.
\subsection{Notation} Throughout the paper, for every integer $k\in\mn$, we set $[k]=\{1,\ldots,k\}$. For every $q\in[0,\infty]$, we
denote by $\|\bu\|_q$ the usual $\ell_q$-(quasi)norm of a vector $\bu\in\mr^k$, that is
$$
\|\bu\|_q =
\begin{cases}
\text{Card}(\{j:u_j\not=0\}), &q=0, \\
(\sum_{j\in [k]}|u_j|^q)^{1/q},& 0<q<\infty,\\
\max_{j\in[k]} |u_j|,& q=\infty.
\end{cases}
$$
For any set $T\subset [p]$, we denote by $T^c$ and $|T|$ the complementary set $[p]\setminus T$
and the cardinality of $T$, respectively. For every matrix $\bfA\in \mr^{p\times q}$ and any subset $T$ of $[q]$, we denote
by $\bfA_T$ the matrix obtained from $\bfA$ by removing all the columns belonging to $T^c$. For a vector $\bu\in\mr^p$ and a
set $T\subset [p]$, $\bu_T$ is the vector obtained from $\bu$ by removing all the coordinates belonging to $T^c$.
The transpose and the Moore-Penrose pseudoinverse of a matrix $\bfA$ are denoted by $\bfA^\top$ and $\bfA^\dag$,
respectively. For two vectors $\bu$ and $\bu'$ of the same dimension $p$, we define $\odot$ as the coordinatewise
product, that is $\bu\odot\bu' = (u_1u_1',\ldots,u_p u_p')^\top$. We write $\mathbf 1_p$ (resp.\ $\mathbf 0_p$)
for the vector of $\mr^p$ having all coordinates equal to one (resp.\ zero). For the design matrix $\bfX$ and any subset $T$
of $[p]$, we denote by $V_T$ the linear subspace of $\mr^n$ spanned by the columns of $\bfX_T$. Further, we denote
by $\Pi_T$ the orthogonal projector onto $V_T$ and by $\rho_T $ the maximal Euclidean distance between the
normalized columns of $\bfX$ and the set $V_T$, that is,
\begin{equation}\label{mainquantity}
\rho_T :=\max_{j\in[p]} \min_{\bv\in V_T}\|\bv-n^{-1/2}\bx^j\|_2 = n^{-1/2}\max_{j\in[p]} \|(\bfI_n-\Pi_T)\bx^j\|_2.
\end{equation}
For two vectors $\bbeta,\bbeta'\in\mr^p$, we denote by $\ell_n(\bbeta,\bbeta')$ the prediction loss $\frac1n\|\bfX(\bbeta-\bbeta')\|_2^2$.
In all the asymptotic considerations, we will write $a_n\lesssim b_n$ for two positive sequences $(a_n)$ and $(b_n)$ when
for some $c\in(0,\infty)$ it holds that $\varlimsup_{n\to\infty} (a_n/b_n) \le c$. Further, we write $a_n\asymp b_n$
for two sequences $(a_n)$ and $(b_n)$ which are of the same order, that is $a_n\lesssim b_n\lesssim a_n$.

\subsection{Outline of the paper} The rest of this work is organized as follows.
The next section presents some new risk bounds for the prediction risk of the Lasso under no condition on the covariates.
These results provide an answer to the first question above. Section~\ref{sec:fast} is devoted to some refinements of the sharp sparsity oracle inequalities
with fast rates based on compatibility factors\cnew{~\cite{SunZhang}}. They imply, in particular, that the total variation estimator of piecewise
constant signals is nearly rate optimal. We present in Section~\ref{limits} an example
showing that for some particularly unfavorable design matrices it is impossible to get rates faster than $1/\sqrt{n}$, even if
$|J^*|$ is very small.  ``Slow'' rates
that involve the quantity $\rho_T$, accounting for the severity of the correlations within covariates, are
developed in Section~\ref{sec:slow}. In particular, they allow us to answer the second and the fourth questions raised
in the Introduction. Discussion with related work and further remarks
are placed in Section~\ref{sec:discussion}. We summarize the contributions of this work and outline some open questions
in Section~\ref{sec:conclusion}. The proofs of all the results stated in the paper are deferred to Section~\ref{sec:proofs}.

\section{Fast rates for Lasso projections}\label{sec:2}
The goal here is to present some new results concerning the accuracy of the Lasso in terms of the prediction loss
when almost no assumption on the relationship between the covariates is required. In particular, we will show that
the estimator $\bfX\betaL$ of the mean $\bfX\betaz$ of the vector $\by$, when projected on a subspace
of $\mr^n$ spanned by a small number of columns of $\bfX$, achieves fast rates of convergence provided that $\lambda$
is of the order of $n^{-1/2}$. This will be complemented in Section \ref{sec:slow}, were we establish new results
characterizing the so called slow rates for the Lasso and show that, in some circumstances, these rates may be significantly
faster than $(s^*/n)^{1/2}$.

We begin by discussing one of the main points that contrasts our approach with the previous ones used in the literature.
Let $T$ be a subset of $[p]$ which may be the set of relevant covariates or any other set.
Let $\Pi_T = \bfX_T (\bfX_T^\top \bfX_T)^\dag \bfX_T^\top$ be the orthogonal projector onto the subspace
of $\mr^n$ spanned by the columns of $\bfX_T$. An idea underpinning our results below is that when only noisy
observations of the vector $\bfX\betaz$ are available, it is practically impossible to make the difference
between the true vector $\betaz$ and the vector $\betazT$ defined by the relations
$$
\betazT_T = \betaz_T+(\bfX_T^\top\bfX_T)^\dag\bfX_T^\top\bxi\qquad\text{and}\qquad \betazT_{T^c} = \betaz_{T^c}.
$$
In fact, one easily checks that
\begin{equation}\label{y}
\by= \bfX_T\betaz_T+\bfX_{T^c}\betaz_{T^c}+\bxi = \bfX\betazT+(\bfI_n-\Pi_T)\bxi.
\end{equation}
When the rank of $\Pi_T$ is much smaller than the sample size $n$, the noise vectors $\bxi$ and $(\bfI_n-\Pi_T)\bxi$
exhibit similar behavior. Therefore, both $\betaz$ and $\betazT$ may be seen as the signal part of the noisy observation $\by$.
In what follows, we exploit this idea in order to establish oracle inequalities\footnote{We refer the reader to \cite{BTW07} for an
introduction to sparsity oracle inequalities.} on the prediction error
$\ell_n(\betaL,\betazT)$ and some other related quantities. Since $\betazT$ is merely a
perturbation of $\betaz$, all the bounds proved for $\ell_n(\betaL,\betazT)$ carry over
similar bounds on the conventional prediction loss $\ell_n(\betaL,\betaz)$.

\begin{theorem}\label{thm1}
Let $T$ be any subset of $[p]$ and let\footnote{It follows from the Cauchy-Schwarz inequality that $\nu_T$ is not smaller than the
smallest singular value of the matrix $\frac1{\sqrt n}\bfX_T$.} $\nu_T=\inf_{\bu\in\mr^{|T|}}\frac{\sqrt{|T|}\cdot\|\bfX_T\bu\|_2}{\sqrt{n}\|\bu\|_1}$.
For every $\lambda>0$, it holds that
$$
\frac1n\|\Pi_T\bfX(\betaL-\betazT)\|_2^2 \le \frac{\lambda^2|T|}{\nu_T^2}.
$$
\end{theorem}

A remarkable fact is that the claim of the foregoing theorem is valid under very weak assumptions on the design matrix $\bfX$,
for every value of the tuning parameter $\lambda>0$ and whatever the noise vector $\bxi$ is.
An immediate consequence of this result that follows from the triangle inequality is
\begin{equation}\label{eq:2.2}
\frac1{\sqrt{n}}\|\Pi_T\bfX(\betaL-\betaz)\|_2 \le \bigg(\frac{\lambda \sqrt{|T|}}{{\nu_T}} +\frac{\|\Pi_T\bxi\|_2}{\sqrt{n}}\bigg),\qquad \forall
T\subset [p].
\end{equation}
Note that the vector $\Pi_T\bxi/\sqrt{n}$ appearing in the last term in this inequality is exactly equal to the
stochastic error of the least squares estimator when only the covariates $\{\bx^j:j\in T\}$ are considered as relevant.
The Euclidean norm of this vector is typically of the order of $\sigma^*\sqrt{|T|/n}$ and represents a lower bound
on the risk when no information other than $|T|$-sparsity of $\betaz$ is available.
Since it is usually recommended to choose $\lambda$ not larger than $\sigma^*\sqrt{2\log(p/\delta)/n}$, for some prescribed tolerance level
$\delta\in(0,1)$, we conclude that  $\Pi_T\bfX\betaL$ estimates the vector $\Pi_T\bfX\betaz$ with the fast
rate of convergence ${\sigma^*}^2{|T|\log(p)/n}$.

%Another straightforward corollary of Theorem~\ref{thm1}---following from the relation $\bfX\betazT = \bfX\betaz+\Pi_T\bxi$---is that sharp
%oracle inequalities established for the prediction loss $\ell_n(\betaL-\betazT)\|_2^2$ lead to oracle inequalities for $\ell_n(\betaL,\betaz)$
%which are still sharp but contain an additional small remainder term. More precisely,
%\begin{align}\label{compare}
%\ell_n(\betaL,\betaz) & = \frac1n\|\Pi_T\bfX(\betaL,\betaz) +\frac1n\|(\bfI_n-\Pi_T)\bfX(\betaL,\betaz)\nonumber\\
%&=\frac1n\|\Pi_T\bfX(\betaL-\betazT)-\Pi_T\bxi\|_2^2 +\frac1n\|(\bfI_n-\Pi_T)\bfX(\betaL-\betazT)\|_2^2\nonumber\\
%&\le \ell_n(\betaL-\betazT)\|_2^2+\frac{2\|\Pi_T\bxi\|_2^2}{n}+\frac{|T|\,\lambda^2}{\nu_T^2}.
%\end{align}
%Since in most commonly used situations $\lambda$ is of the order of $\sqrt{\log(p)/n}$, the remainder terms in the above inequality
%are $O(|T|\log(p)/n)$, which is nearly optimal for sparsity oracle inequalities.

Relation (\ref{eq:2.2}) also demonstrates that the prediction loss of the Lasso decreases to zero at the fast rate of convergence  $s\log(p)/n$
in some particular cases with strongly correlated covariates. This result is stated in the following proposition.

\begin{proposition}\label{prop1}
If there is a subset $T$ of $[p]$ of cardinality $s$ such that all the covariates $\{\bx^j:j\in T^c\}$ belong to the linear span of
$\{\bx^j:j\in T\}$, then for every $\lambda>0$
\begin{equation}\label{eq:2.3}
\ell_n(\betaL,\betaz)^{1/2} \le \frac{\lambda \sqrt{s}}{{\nu_T}} +\frac{\|\Pi_T\bxi\|_2}{\sqrt{n}}.
\end{equation}
In particular, for every vector $\bxi$ with uncorrelated entries such that $\bfE[\bxi]=0$ and $\max_i\bfE[\xi_i^2]\le \sigma^*{}^2$,
\begin{equation}\label{eq:2.4}
\bfE[\ell_n(\betaL,\betaz)] \le \frac{2\lambda^2 s}{\nu_T^2} +\frac{2\sigma^*{}^2s}{n}.
\end{equation}
If, in addition, $\bxi\sim \sigma^*\mathcal N_n(0,\bfI_n)$, then with probability at least $1-\delta$
\begin{equation}\label{eq:2.5}
\ell_n(\betaL,\betaz)\le \frac{2\lambda^2 s}{\nu_T^2} +\frac{4\sigma^*{}^2 (s+2\log(1/\delta))}{n}.
\end{equation}
\end{proposition}

The first two claims of this proposition trivially follow from (\ref{eq:2.2}), while the third claim follows from (\ref{eq:2.3})
using the fact that $\|\Pi_T\bxi\|_2^2$ is drawn from the chi-squared distribution $\chi^2_s$ in conjunction with the well-known
results on the tails of the latter.

This proposition answers to the first question raised in the introduction concerning the performance of the Lasso as a function of the rank of $\bfX$
when the latter is small as compared to $n$. In fact, let us denote by $\bar\nu_r$ the maximal value of $\nu_T$ over all possible subsets of
$[p]$ of cardinality $r=\rk(\bfX)$: $\bar\nu_r = \max_{T:|T|=r} \nu_T$.
It follows from (\ref{eq:2.5}) that when $\bxi$ is Gaussian and $\lambda = \sqrt{2\log(p)/n}$, for every $\delta\in(0,1)$, with probability $1-\delta$,
\begin{equation}\label{eq:2.51}
\ell_n(\betaL,\betaz) \le \frac{4\log(p) \rk(\bfX)}{n\bar\nu_r^2} +\frac{4\sigma^*{}^2 (\rk(\bfX)+2\log(1/\delta))}{n}.
\end{equation}

%Finally, let us note that for subsets $T$ of cardinality one, we have $\nu_T=1$ and, therefore, inequality (\ref{eq:2.2}) takes the
%well-known form
%\begin{equation}
%\frac1{n}|(\bx^j)^\top\bfX(\betaL-\betaz)| \le \lambda +\frac{|\bxi^\top\bx^j|}{n},\qquad \forall j\in[p],
%\end{equation}
%which is a simple consequence of the Karush-Kuhn-Tucker conditions for the Lasso.

\newcommand{\sectionlimits}{\section{Limits of fast rates: an example}\label{limits}
In this section, we show that the prediction loss of the Lasso is in some cases at best of the order of $n^{-1/2}$, whatever the tuning parameter is. This example provides - to the best of our knowledge - the first proof that in some cases, the Lasso can not achieve fast rates even if the regression vector $\beta^*$ has fixed length.

\begin{example}\label{example1}
Let $n\ge 2$ be an integer. We set $m$ to be the largest integer less than $\sqrt{2n}$ and define the design matrix $\bfX\in\mr^{n\times 2m}$ by
$$
\bfX = \sqrt{\frac{n}{2}}
\begin{pmatrix}
\mathbf 1_m^\top & \mathbf 1_m^\top\\
\bfI_m & -\bfI_m\\
\mathbf 0_{(n-m-1)\times m} &\mathbf 0_{(n-m-1)\times m}
\end{pmatrix}.
$$
If we denote by $\{\be_j:j\in[n]\}$ the canonical basis of $\mr^n$,
the columns of this matrix are of the form $\bx^j = \sqrt{n/2}\,(\be_1+\be_{j+1})$ for $j=1,\ldots,m$ and
$\bx^j = \sqrt{n/2}\,(\be_1-\be_{j-m+1})$ for $j=m+1,\ldots,2m$. To avoid unnecessary technicalities, we assume in this example
that the noise vector is composed of i.i.d.\ Rademacher random variables, that is $\bfP(\bxi = \bfs) = 2^{-n}$ for every
$\bfs\in\{\pm 1\}^n$ (thus $\sigma^*=1$). Let the true regression vector be $\betaz\in\mr^{2m}$ such that $\beta^*_1=\beta^*_{m+1}=1$ and
$\beta^*_j=0$ for every $j\in[2m]\setminus\{1,m+1\}$.

\begin{proposition} \label{prop:2}
For any $\lambda>0$, the prediction loss of the Lasso $\betaL$ satisfies the inequality
$$
\bfP\Big(\ell_n(\betaL,\betaz) \ge \frac1{2\sqrt{2n}}\Big)\ge \frac12.
$$
\end{proposition}

There are at least three reasons that make this example particularly instructive.
First, it shows that the correlations between the covariates
need not to be close to $\pm 1$ to cause the failure of the fast rates.
Even in the case of small fixed correlations the rate of convergence
of the Lasso in prediction loss may be not smaller than $Cn^{-1/2}$. Second, the foregoing result is true for every $\lambda>0$. Thus, even an
oracle choice of $\lambda$ cannot prevent slow rates. Third, it is valid for a small value of sparsity index: the $\ell_0$-norm of $\betaz$
is equal to 2. In the literature, other examples on which the Lasso fails to achieve fast rates have been proposed (see Section 2 in \cite{CandesPlan09}),
however, to the best of our knowledge, this is the first counter-example in which such a result is analytically proved for fixed sparsity, fixed
correlations, any value of $\lambda$ and a $\betaz$ independent of $n$.
\end{example}

This example clearly demonstrates the limits of the Lasso as a method of prediction. While for several other prediction
procedures~\cite{DT07,RigTsy11,DT12a,DT12b} fast rates are valid without any condition on the correlations between the predictors, some
relatively strong assumptions are necessary for the Lasso to achieve fast rates. It should be noted in defense of the Lasso that it
presents major advantages in terms of computational complexity.}

\section{Fast rates under relaxed compatibility assumptions on the design matrix}\label{sec:fast}

%\subsection{Oracle inequalities of Sun and Zhang \cite{SunZhang}}

To the best of our knowledge, the sharpest oracle inequality for the  Lasso available in the
literature is the one presented in \cite{SunZhang}. We begin by stating their result\footnote{The results
stated below do not match exactly with those stated in \cite{SunZhang}, but they can be easily deduced from
the proofs in \cite{SunZhang}} in order to discuss what can be learnt from it concerning the questions presented
in the introduction. Then, we state a new oracle inequality that combines the proof of \cite{SunZhang} and
the idea of estimating $\betazT$ instead of $\betaz$ in order to get some improvements.

For every set $T\subset[p]$
and any $\bar c>0$ we recall the definition of the compatibility factor $\kappa_{T,\bar c}\ge 0$:
\begin{equation}\label{kappa1}
\kappa_{T,\bar c} = \inf_{\bdelta\in\mr^p : \|\bdelta_{T^c}\|_1 < \bar c \|\bdelta_T\|_1} \frac{|T|\cdot\|\bfX\bdelta\|_2^2}
{n(\|\bdelta_T\|_1-\bar c^{-1} \|\bdelta_{T^c}\|_1)^2}.
\end{equation}
\begin{theorem}[\cite{SunZhang}, Theorem 4]\label{thmSZ}
Let $\delta\in(0,1)$ be a fixed tolerance level.  If for some $\gamma>1$, the tuning parameter of the Lasso
satisfies $\lambda=\gamma \sigma^*\big(\frac2n\log(p/\delta)\big)^{1/2}$, then  with probability at least $1-\delta$,
\begin{align*}
\ell_n(\betaL\!\!,\betaz)
&\le \inf_{\barbbeta\in\mr^p,T\subset [p]} \bigg\{ \ell_n(\barbbeta,\betaz) + 4\lambda \|\barbbeta_{T^c}\|_1 +\frac{2(1+\gamma)^2\sigma^*{}^2 |T|\log(p/\delta)}{n\kappa_{T, (\gamma+1)/(\gamma-1)}} \bigg\}.
\end{align*}
\end{theorem}

An important feature of this inequality is its sharpness, reflected by the fact that the constant
in front of the infimum, often referred to as the leading constant of an oracle inequality (OI), is equal to one. The first sharp OI
with fast rate of convergence of the remainder term has been proved in \cite{KLT}. It was then refined and extended to the procedure
square-root Lasso (also known as the scaled Lasso) in \cite{SunZhang}.

Let us state now some refinements of Theorem~\ref{thmSZ}.
For any subset $T$ of $[p]$, let us introduce the weights\footnote{In the definition of $\bar\bomega$, we use the convention $0/0=0$.}
\begin{align}
\omega_j(T,\bfX) &= \frac1{\sqrt{n}}\;{\|(\bfI_n-\Pi_T)\bx^j\|_2},\qquad%\forall j\in[p],\label{weights1}\\
\bar\omega_j(T,\bfX) = \frac{\omega_j(T,\bfX)}{\max_{\ell\in [p]} \omega_\ell(T,\bfX)},\qquad\forall j\in[p].\label{weights2}
\end{align}
Since $\bx^j$ are normalized to have an $\ell_2$ norm at most equal to $\sqrt{n}$, the weights $\omega_j(T,\bfX)$ are all between
zero and one. Furthermore, they vanish whenever $\bx^j$ belongs to the linear span of $\{\bx^\ell,\ell\in T\}$. In particular,
$\omega_j(T,\bfX)=0$ for every $j\in T$.  Using these weights and any $\gamma>0$, we define the sets
\begin{align*}
\mathcal C_0(T,\gamma,\bomega) &=\Big\{\bdelta\in\mr^p: \|(\mathbf 1_p-\gamma^{-1}\bomega)_{T^c}\odot\bdelta_{T^c}\|_1 <
\|\bdelta_{T}\|_1\Big\}.
\end{align*}
When $\bomega=\mathbf 1_p$, we write $\mathcal C_0(T,\gamma)$ instead of $\mathcal C_0(T,\gamma,\bomega)$.
%and $\mathcal C_1(T,\gamma,\bomega) =\big\{\bdelta\in \mathcal C_0(T,\gamma,\bomega) :
%\bdelta_{T^c}^\top\bfX_{T^c}^\top\bfX\bdelta<0\big\}$.
\begin{definition}[Compatibility factors]
For every vector $\bomega\in\mr^p$ with nonnegative entries, we call the weighted compatibility factor the quantity
$$
\bar\kappa_{T,\gamma,\bomega} = \inf_{\bdelta \in \mathcal C_0(T,\gamma,\bomega)}
\frac{|T|\cdot\|\bfX\bdelta\|_2^2}{n\big\{\|\bdelta_{T}\|_1-\|(\mathbf 1_p-\gamma^{-1}\bomega)_{T^c}\odot\bdelta_{T^c}\|_1\big\}^2}.
$$
%The first weighted compatibility factor is
%$$
%\bar\kappa^{(1)}_{T,\gamma,\bomega} = \inf_{\bdelta \in \mathcal C_1(T,\gamma,\bomega)}
%\frac{|T|\cdot\|\bfX\bdelta\|_2^2}{n\big\{\|\bdelta_{T}\|_1-\|(\mathbf 1_p-\gamma^{-1}\bomega)_{T^c}\odot\bdelta_{T^c}\|_1\big\}^2}.
%$$
\end{definition}

The weighted compatibility factors with weights $\bomega$ and $\bar\bomega$ defined in (\ref{weights2}) are particularly useful for explaining the accuracy of the Lasso as measured by the prediction loss. They relax the assumptions
previously known in the literature that lead to fast rates.

\begin{theorem}\label{thm3}
Let $\delta\in(0,1)$ be a fixed tolerance level. If for some value $\gamma>1$, the tuning parameter of the Lasso
satisfies $\lambda=\gamma \sigma^*\sqrt{2\log(p/\delta)/n}$, then  on an event of probability at least $1-2\delta$,
the following bound holds:
\begin{align}
\ell_n(\betaL\!\!,\betaz)
&\le \inf_{\barbbeta\in\mr^p,T\subset [p]} \Big\{ \ell_n(\barbbeta,\betaz) + 4\lambda \|\barbbeta_{T^c}\|_1 +
\frac{4\sigma^*{}^2|T|\log(p/\delta)}{n}\cdot r_{n,p,T}\Big\},\label{eq:3.1}
\end{align}
where the remainder term is given by $r_{n,p,T}=\log^{-1}(p/\delta)+
2|T|^{-1}+\gamma^2\bar\kappa_{T, \gamma, \bomega}^{-1}$.
%On the same event, we have
%\begin{align}
%\ell_n(\betaL\!\!,\betaz)
%&\le \frac{4\sigma^*{}^2 |J^*|\log(p/\delta)}{n}\cdot r_{n,p,J^*}^{(1)},\label{eq:3.2}
%\end{align}
%where $r_{n,p,J^*}^{(1)}$ is defined by the same formula as $r_{n,p,J^*}$ except that $\bar\kappa_{J^*, \gamma, \bomega}$
%is replaced by $\bar\kappa^{(1)}_{J^*, \gamma, \bomega}$.
Furthermore, if for some $T\subset [p]$ and some $\gamma>1$, $\lambda=\gamma \sigma^*\rho_T
\sqrt{2\log(p/\delta)/n}$, then  with probability at least $1-2\delta$, the following bound holds
\begin{align}
\ell_n(\betaL\!\!,\betaz)
&\le \inf_{\barbbeta} \!\Big\{ \ell_n(\barbbeta,\betaz) + 4\lambda \|\barbbeta_{T^c}\|_1\!\Big\} +
\frac{4\sigma^*{}^2\rho_T ^2|T|\log(p/\delta)}{n}\cdot \bar r_{n,p,T},\label{eq:3.3}
%\\
%\ell_n(\betaL\!\!,\betaz)
%&\le \frac{4\sigma^*{}^2\rho_{J^*}^2 |J^*|\log(p/\delta)}{n}\cdot \bar r_{n,p}^{(1)},\label{eq:3.4}
\end{align}
where the remainder term is given by $\bar r_{n,p,T}=
\frac{(1+2|T|^{-1}\log(1/\delta))}{\rho_T ^2\log(p/\delta)}+\frac{\gamma^2}{\bar\kappa_{T, \gamma, \bar\bomega}}$.
%and $\bar r_{n,p}^{(1)}$ is defined by the same formula as $\bar r_{n,p}$ except that $\bar\kappa_{J^*, \gamma, \bar\bomega}$
%is replaced by $\bar\kappa^{(1)}_{J^*, \gamma, \bar\bomega}$.
\end{theorem}

The main difference between inequalities (\ref{eq:3.1}) and (\ref{eq:3.3}) is the presence
of the factor $\rho_T ^2$ in the numerator of the last term. This factor is always not larger than $1$. However, in order to introduce it we
needed to replace the compatibility factor $\bar\kappa_{T, \gamma,\bomega}$ by $\bar\kappa_{T, \gamma, \bar\bomega}$ and to
deflate $\lambda$ by the factor $\rho_T $. From a practical point of view, this last modification is not always easy to implement, since
the quantity $\rho_T $ depends on the set $T$ which can be thought of as the best possible set of covariates. This set being unknown, the
claim of (\ref{eq:3.3}) is to be interpreted as a theoretical justification for choosing the tuning parameter smaller than the
universal value. Such a choice can be made, for instance, by cross validation. It is also possible to perform a sparse PCA on the
set of covariates in order to choose a suitable value of $\lambda$ (cf.\ Section~\ref{ss:5.2} for more details).

\begin{example}[Total variation penalty for piecewise constant functions]\label{exTV}
In image denoising and signal processing, total variation type penalties are often employed to enforce similarity between neighboring
pixels or values of the signal. In the one-dimensional setting, the problem may be formulated as follows. Assume that a piecewise constant function
$f^*:[0,1]\to\mr$ is observed on the regular grid in a noisy environment: $y_i=f^*(i/n)+\xi_i$, for $i=1,\ldots,n$. Let us denote the unknown vector
of values of $f^*$ on the grid by $\boldf^*  = (f^*(1/n),\ldots,f^*(1))^\top$ and define the total variation penalty of a vector $\boldf\in\mr^n$
by $\|\boldf\|_{\rm TV} = \sum_{i=1}^n |f_{i}-f_{i-1}|$ with the convention that $f_0=0$. Then, the TV-penalized least squares estimator of $\boldf^*$ is defined as
\begin{align}\label{eq:TV}
\hat\boldf{}^{\rm TV} \in \text{arg}\min_{\boldf\in\mr^n} \Big\{\frac1n\|\by-\boldf\|_2^2+\lambda\|\boldf\|_{\rm TV}\Big\},
\end{align}
where $\lambda>0$ is a tuning parameter. This estimator, hereafter referred to as TV-estimator,
has been shown to be closely related to the Lasso \cite{Zaid07,Zaid10}. More precisely, if we define the vector of differences $\bbeta\in\mr^n$
by $\beta_j=f_j-f_{j-1}$, then (\ref{eq:TV}) is equivalent to (\ref{lasso}) with the $n\times n$ design matrix $\bfX=(\mathds 1(i\ge j))_{i,j}$.
Despite its popularity in applications, it is very surprising that the TV-estimator and, more precisely, its prediction accuracy
is not yet completely understood. In a purely asymptotic setting $n\to\infty$, \cite{MammenvdG} establish that the $\ell_2$-loss of
$\hat\boldf{}^{\rm TV}$, defined by $\frac1n\|\hat\boldf{}^{\rm TV}-\boldf^*\|_2^2$, is of the order of $n^{-2/3}$. This is, however, just an upper bound
on the risk of $\hat\boldf{}^{\rm TV}$, and this upper bound is much worse than the optimal rate of convergence, known to be $n^{-1}$ in the problem of
estimating piecewise constant functions. This optimal rate is achieved, for instance, by the penalized least squares with a penalty proportional
to the number of jumps, \textit{i.e.}, the $\ell_0$-norm of the difference vector (see \cite{Munk} and the references therein). The question whether
it is possible or not to improve the rate $n^{-2/3}$ for the TV-penalized estimator and, eventually, to achieve the minimax rate, has remained open so far.

On the other hand, more recent papers \cite{Zaid07,Zaid10} propose nonasymptotic risk bounds for $\hat\boldf{}^{\rm TV}$. Without any
assumption, they show that for $\lambda\asymp n^{-1/2}$ their risk bound is of the order of $n^{-1/2}$. They also notice that if the
TV-estimator with $\lambda\asymp n^{-1}$ has only a few jumps, then its $\ell_2$-loss is of the optimal order $n^{-1}$. This result is,
however, not very satisfactory since reducing $\lambda$ down to the order $n^{-1}$ is quite likely to significantly increase the number of
jumps in the TV-estimator. Put differently, there is no theoretical result assessing the probability of getting only a few jumps when
$\lambda\asymp n^{-1}$. This raises some new questions: (a) Is it possible to establish sharp oracle inequalities for TV-estimator
with optimal rate of convergence? (b) Is it really necessary to choose
$\lambda$ very small for achieving the optimal rate? (c) What is the rate of convergence in terms of the number of jumps, when the latter is allowed to
increase with $n$? In order to show that the theoretical tools developed in this section provide almost exhaustive answers to these questions, we need the
following result.

\begin{proposition}\label{prop:3.1}
Let $\bfX$ be the $n\times n$ matrix with entries $x_{ij}=\mathds 1(i\ge j)$ and let $\ba\in\mr_+^n$ be a given vector of ``weights''.
For every $T=\{j_1,\ldots,j_s\}\subset [n]$ and for every $\bu\in\mr^n$, we have
$$
\|\bu_T\odot \ba_{T}\|_1-\|\bu_{T^c}\odot \ba_{T^c}\|_1 \le
4\|\bfX\bu\|_2 \bigg(2\sum_{j\in[n]}|a_j-a_{j+1}|^2+2(s+1)\|\ba\|_\infty^2\Delta_{\min,T}^{-1}\bigg)^{1/2},
$$
where $\Delta_{\min,T} = \min_{\ell\in[s+1]}|j_{\ell+1}-j_\ell|$ with the convention $j_0=1$ and $j_{s+1}=1$.
\end{proposition}

The proof of this result, deferred to Section~\ref{sec:proofs}, is carried out using a completely new approach based on
a probabilistic argument. We believe that this argument may be used in other situations for evaluating the compatibility factors theoretically.
All the previous efforts for evaluating the compatibility and restricted eigenvalue constants, focused on weakly correlated designs
(see, for instance, \cite{BRT}). In contrast with this, our approach provides bounds on compatibility factors even for strongly
correlated designs. Indeed, many pairs of columns of matrix $\bfX$ corresponding to the TV-estimator have correlation of the order of $1-n^{-1/2}$.

For applying Proposition~\ref{prop:3.1} to the TV-estimator, we choose $a_j = 1$ for every $j$ belonging to the set $T$, which presumably contains
the jumps of $f^*$, and  $a_j = 1-\frac1{2\sqrt{n}}\|(\bfI_n-\Pi_{T})\bx^j\|_2$, $j\in T^c$. In what follows, we denote by $\Delta_{\min,T}$
the smallest distance between two jumps, that is $\Delta_{\min,T}=\min_{\ell\in[s+1]}|j_{\ell}-j_{\ell-1}|$ with the convention that $j_0=1$ and $j_{s+1}=n+1$.

\begin{proposition}\label{prop:3.2}
Let $\boldf^*$ be a piecewise constant vector and $J^*=\{j\in[n]:f^*_j\not=f^*_{j+1}\}$. If the tuning parameter
satisfies $\lambda=2\sigma^*\{(2/n)\log(n/\delta)\}^{1/2}$, then  on an event of probability at least $1-2\delta$,
the following bound holds for every  nonempty $T\subset [n]$:
\begin{align}\label{eq:riskTV}
\frac1n\|\hat\boldf{}^{\rm TV}-\boldf^*\|_2^2
&\le \inf_{\bar\boldf\in\mr^n} \bigg\{ \frac1n\|\bar\boldf-\boldf^*\|_2^2 + 4\lambda \|\bar\boldf_{T^c}\|_{\rm TV}\bigg\} +
\frac{4\sigma^*{}^2|T|\log(n/\delta)}{n}\cdot r_{n,T},
\end{align}
where the bounded remainder term is given by $r_{n,T}= {3}+{256}(\log(n)+(n/\Delta_{\min,T}))$.
\end{proposition}

The risk bound (\ref{eq:riskTV}) drastically improves the results on the $\ell_2$-loss of the TV-estimator ever proved in the
literature. Not only it holds for finite samples, is with a leading constant one and provides a risk bound valid with high probability,
but, more importantly, it has a remainder term of the order of $|T|(\log(n))^2/n$. The rate of decay to zero of this term when $n\to\infty$
is much faster than what was known before and should be contrasted with $n^{-2/3}$ established in \cite{MammenvdG}. More precisely, when
the true function $\boldf^*$ is piecewise constant on a partition of $s$ intervals, taking in (\ref{eq:riskTV}) $\bar\boldf=\boldf^*$ and
$T=\{j\in[n]:f^*_j\not=f^*_{j-1}\}$, the terms in accolades at the right hand-side vanish and one gets the inequality
\begin{align}\label{eq:riskTV1}
\bfP\Big(\frac1n\|\hat\boldf{}^{\rm TV}-\boldf^*\|_2^2
&\le \frac{4\sigma^*{}^2|J^*|\log(n/\delta)}{n}\cdot \big({3}+{256}(\log(n)+(n/\Delta_{\min,J^*}))\big)\Big)\ge 1-2\delta.
\end{align}
When the vector $\boldf\in\mr^n$ consists of the values of a piecewise constant function $f$ at the points $\{i/n, i\in[n]\}$, of the regular grid,
the term $n/\Delta_{\min,J^*}$ is bounded by a constant (for $n\to\infty$ and fixed $f$). In this case,  the upper bound in (\ref{eq:riskTV1})
is of the nearly optimal order. Furthermore, risk bound (\ref{eq:riskTV1}) holds for every $|J^*|$, even if it tends to infinity with $n$.
To the best of our knowledge, this is  the first result of this type. All the previous asymptotic results considered the number of jumps $|J^*|$
as fixed. Moreover, our result is valid for the universal choice of the tuning parameter and not the very small one evoked in \cite{Zaid10}.
To complete this discussion, let us mention that the constant 256 in (\ref{eq:riskTV1}) is definitely sub-optimal and it is out of scope of this
work to look for the best possible constants.
\end{example}

\sectionlimits
\section{``Slow'' rates accounting for high correlations}\label{sec:slow}

In the preceding section, we have discussed fast rate bounds, that is, bounds that contain the tuning parameters to the power two. In this section, we turn to slow rate bounds, that is, bounds that contain the tuning parameters to the power one. We present slow rate bounds that entail---in contrast to what the nomenclature suggests---fast rates if the correlations are properly incorporated into the tuning parameters.  {These results considerably extend  and sharpen insights obtained in~\cite{HebiriLederer,vdGeer11} and are of particular interest for the Least-Squares estimator with total variation penalty (TV-estimator).}  We can deduce, in particular, that the TV-estimator is almost minimax for the estimation of monotone or H\"older continuous signals and, therefore, improve on results in~\cite{MammenvdG}, for example.
% \textcolor{red}{As mentioned in previous section, a risk bound providing slow rates for the prediction loss of the Lasso---see equation (\ref{eq:slow1})---can be
% deduced from the proof of \cite{SunZhang}. This bound improves on the analogous results proved in earlier papers \cite{Barron,MM11}.
% Yet, it is valid only for values of the tuning parameter $\lambda$ that are larger than  $\sqrt{(2/n)\log(p/\delta)}$. In some situations, this
% condition is not met for the tuning parameter minimizing the prediction loss. Even worse, in some cases choosing $\lambda\asymp \sqrt{\log(p)/n}$
% leads to suboptimal rates of convergence. We propose in this section one solution for alleviating the condition on $\lambda$ and show the consequences
% of the obtained risk bound in the case of prediction with total variation type penalties.}

The following slow rate bound is the main result of this section.
%Reformulated Thm 4. Suppressed.
\begin{theorem}\label{thm:4}
Let $T\subset [p]$ be a set of indices and let $\delta>0,~\gamma\geq 1$ be constants. Then,
if the tuning parameter $\lambda$ is not smaller than $\gamma\sigma^*\rho_{T}\sqrt{2\log(p/\delta)/n}$, the Lasso~\eqref{lasso} fulfills
\begin{align*}%\label{eq:4.1}
\ell_n(\betaL\!\!,\betaz)+\frac{2(\gamma{-1})\lambda}{\gamma}\|\betae\|_1
&\le \inf_{\barbbeta\in\mr^p} \Big\{ \ell_n(\barbbeta,\betaz)  + \frac{2(\gamma+1)\lambda}{\gamma} \|\barbbeta\|_1\Big\}%\nonumber\\
%&\qquad\qquad
+ \frac{2\sigma^*{}^2(|T|+2\log(1/\delta))}{n}
\end{align*}
with probability at least $1-2\delta$.
\end{theorem}
%\jl{I have suppressed all equation numbers that were nowhere cited.}
% %Reformulated Thm 4. Suppressed.
% \begin{theorem}\label{thm:4}
% Let $T\subset [p]$ be a set of indices and $\delta>0$ a positive constant. If the tuning parameter $\lambda$ is larger or equal to $\rho_{T}\sigma^*\sqrt{\frac{2\log(p/\delta)}{n}}$, it holds that \begin{align}\label{eq:4.1}
% \ell_n(\betaL\!\!,\betaz)
% &\le \inf_{\barbbeta\in\mr^p} \bigg\{ \ell_n(\barbbeta,\betaz)  + 4 \lambda \|\barbbeta\|_1\bigg\} + \frac{2\sigma^*{}^2(|T|+2\log(1/\delta))}{n}
% \end{align}
% with probability at least $1-2\delta$.
% \end{theorem}

The factor $\rho_T$ (defined in Equation~\eqref{mainquantity}) in the lower bound for the tuning parameter $\lambda$ makes this theorem particularly interesting. Slow rate bounds can be derived using the proofs in~\cite{SunZhang}, but they involve tuning parameters of order $\sqrt{\log(p)/n}$. Theorem~\ref{thm:4} allows for considerably smaller tuning parameters if the variables are correlated; this then  leads to rates in between the classical slow rates and fast rates of order (neglecting logarithmic factors) $\sqrt{s/n}$ and  ${s/n}$, respectively. Theorem~\ref{thm:4} implies in particular fast rates for highly correlated covariates:

%$(c_n)_{n\geq 1}>0$ be a converging sequence with a sufficiently large limit.
\begin{corollary}\label{cor:slow}
Assume that $T_n\subset [p]$ is as set of indices (that may depend on the sample size $n$) such that all covariates $\{\bx^j:j\in [p]\}$ are very close to the linear span of the set of vectors $\{\bx^j:j\in T_n\}$ in the sense that $\rho_{T_n} \lesssim n^{-r}$ for a positive constant $r>0$. Then, if the tuning parameter satisfies $\lambda\ge c\sigma^*\sqrt{{\log(p)/n^{2r+1}}}$ for a sufficiently large constant $c>0$,  the Lasso~\eqref{lasso} fulfills
\begin{align}\label{eq:4.2}
\ell_n(\betaL\!\!,\betaz)
&\lesssim  \bigg(\sqrt{\frac{\log(p)}{n^{2r+1}}}\,\|\betaz\|_1\bigg)\bigvee \frac{|T_n|}{n}
\end{align}
with high probability.\\
If, in particular, the irrelevant covariates $\{\bx^j:j\notin J^*\}$ are within Euclidean distance $1$ of the linear space spanned by the relevant covariates $\{\bx^j:j\in J^*\}$, it holds that $r= 1/2$ and, therefore, the Lasso achieves the fast rate $s/n$ up to logarithmic factors, provided that $\lambda$ is chosen of order ${\sqrt{\log(p)}/n}$ (with sufficiently large constants).
\end{corollary}

\begin{remark}[Effective number of parameters]
  The bound \eqref{eq:4.2} can be further refined replacing the number of parameters $p$ by an effective number of parameters as described in \cite[Section 3.2.2]{HebiriLederer}. This effective number of parameters can be considerably smaller than $p$ if the correlations are high, therefore reducing the bound by a factor up to $\sqrt{\log(p)}$.
\end{remark}

% \begin{remark}[Prior knowledge]
% Using a weighted Lasso, prior knowledge can also be exploited to improve the bound~\eqref{eq:4.2}. Denoting by $\betae^{\text{\rm WL}}$ the weighted Lasso with a tuning parameter as in Corollary~\ref{cor:slow} and weights equal to zero for all indices in $T_n$ and equal to one otherwise, the bound
%  \begin{align*}%\label{eq:4.2adaptive}
% \ell_n(\betae_\lambda^{\text{\rm WL}}\!\!,\betaz)
% &\lesssim  \frac{\left(\sqrt{\frac{\log(p)}{n^{2r-1}}}\,\|\betaz_{T_n^c}\|_1\right)\vee|T_n|}{n}
% \end{align*}
% holds with high probability under the conditions in Corollary~\ref{cor:slow} for the bound~\eqref{eq:4.2}. If in particular $J^*\subset T_n$, the bound becomes
%  \begin{align*}%\label{eq:4.2adaptive2}
% \ell_n(\betae_\lambda^{\text{\rm WL}}\!\!,\betaz)
% &\lesssim  \frac{|T_n|}{n}.
% \end{align*}
% For details, we refer to Section XXX in the Appendix.
% \end{remark}
% \jl{This is a suggestion how to include the part on prior information. What do you think?}

Corollary~\ref{cor:slow} exhibits fast rates for highly correlated but not necessarily perfectly collinear designs. We call a design perfectly collinear if all covariates belong to the linear space spanned by the relevant covariates, that is, $\{\bx^j:j\in [p]\}\subset \text{Span}\{\bx^j:j\in J^*\}$. For these designs, fast rates can be deduced from known results. Corollary~\ref{cor:slow}, in contrast, exhibits fast rates even for designs that differ from perfectly collinear designs by an order of $n^{-1/2}$ (as measured by the maximal distance $\rho_{J^*}$, see Equation~\eqref{mainquantity}). {Thus, Corollary~\ref{cor:slow} is the forth key contribution mentioned in the Introduction.}

The dependence of the tuning parameters on the set $T$ in Theorem~\ref{thm:4} and Corollary~\ref{cor:slow} can lead to additional computational costs. For some applications, such as the total variation penalization discussed below, the set $T$ is completely predetermined. For some other applications, however, the set $T$ is not completely predetermined, promoting the minimization of $\rho_{T}$ over a class of sets (for example,  all sets with a fixed cardinality), which can be computationally expensive. Proposition~\ref{prop:4.1} below %can be considered as interpolation between Theorem~\ref{thm1} and Theorem~\ref{thm:4} that shows that fast rates may be achieved in the scenario $\rho_{T}= O(n^{-1/2})$
%for tuning parameters independent of the set $T$ as well, therefore
provides another risk bound that helps to evade minimizations with respect to $T$ under favorable circumstances. %Proposition~\ref{prop:4.1}
% but is a risk bound instead of an oracle inequality.

\begin{proposition}\label{prop:4.1}
Let $T\subset [p]$ be a set of indices and $\delta>0,~\gamma> 1$ be constants. If the tuning parameter satisfies
$\lambda\ge \gamma\sigma^*\rho_T\sqrt{{2\log(p/\delta)/n}}$, the Lasso~\eqref{lasso} fulfills
\begin{align*}%\label{eq:4.3}
\ell_n(\betaL\!\!,\betaz)
&\le \frac{4\rho_T^2\gamma^2}{(\gamma-1)^2}
\|\betaz\|_1^2+\frac{4\sigma^*{}^2(|T|+2\log(1/\delta))}{n}+ \frac{2|T|\lambda^2}{\nu_T^2}
\end{align*}
with probability at least $1-2\delta$.
\end{proposition}

As before, Proposition~\ref{prop:4.1} shows that the correlations can be exploited adapting the tuning parameter to the design via the measure~$\rho_T$, but, in strong contrast to the above results, Proposition~\ref{prop:4.1} provides fast rate bounds for strongly correlated covariates even for standard, non-adapted tuning parameters of the order $\sqrt{\log(p)/n}$. For example, if $\rho_{J^*}\asymp n^{-1/2}$ and $|J^*|\asymp 1$, fast rates can be deduced from Proposition~\ref{prop:4.1} even with universal tuning parameters of order $\sqrt{\log(p)/n}$; in strong contrast,  considerably smaller tuning parameters of order ${\sqrt{\log(p)}}/{n}$ are required to deduce fast rates from Theorem~\ref{thm:4} for this example. Note, however, that Proposition~\ref{prop:4.1} does not supersede Theorem~\ref{thm:4} in general: for moderate correlations, the first term of the bound in Proposition~\ref{prop:4.1} is large, and Theorem~\ref{thm:4} is then considerably more beneficial.

Theorem~\ref{thm:4} provides, in particular, minimax rates for signal denoising with total variation penalties. In the previous section, we have studied the TV-estimator for piecewise constant signals. In the remainder of this section, we study the TV-estimator for monotone signals (or slightly more general, signals with bounded variation) and for H\"older continuous signals. A recent review on this topic and a detailed analysis of the maximum likelihood estimator in this context can be found in~\cite{Aditya} and an earlier risk bound can be found in~\cite{Zhang02}.
%\jl{I have reorganized the examples somewhat.}

\begin{example}[Predicting monotone functions with the TV-estimator]

%\jl{Reformulated paragraphs. I have suppressed two details: $\alpha\in(0,1]$ and the $\ell_2$ thing. Maybe, we do not have to mention this here in such detail... Also, I think the locations of the results have changed in the meantime.}

In this example, we derive an almost minimax risk bound that is particularly interesting for signals with bounded variation. For this, we apply Theorem~\ref{thm:4} exploiting that the TV-estimator can be considered as a special case of the Lasso.  As mentioned earlier, the TV-estimator~\eqref{eq:TV} corresponds to the Lasso~\eqref{lasso} with the design matrix $\bfX\in\mr^{n\times n}$ with entries $x_{ij}=\mathds 1(i\ge j)$. To transfer the results for the Lasso to the TV-estimator, we assume that $n \geq 3$, fix a positive integer $h\in [n-1]$, denote by $k\geq 2$ the largest integer such that $(k-1)h<n$, and finally set $T=\{1,h+1,2h+1,3h+1,\ldots,(k-1)h+1\}$. The set $T$ induces the partition
\footnote{Here and in the sequel, we use the notation $\llbracket a,b\llbracket:=[a,b[\cap \mn$.}$\{\llbracket 1,h+1\llbracket,\llbracket h+1,2h+1\llbracket,\ldots,\llbracket (k-1)h+1,n+1\llbracket\}$ of $[n]$ with at most $h$ points in the last interval and exactly $h$ points in all other intervals. Moreover, $\Pi_T$ is the orthogonal projection onto the subspace spanned by the vectors that are constant on each of the elements of this partition. This implies $\rho_T=n^{-1/2}\max_j\|(\bfI_n-\Pi_T)\bx^j\|_2 =\max_{j\in[h]} \sqrt{{(j-1)(h-j+1)}/(nh)}\le \sqrt{{h}/(4n)}$. Using $h\leq 2n/k$, we then obtain $\rho_T\leq 1/\sqrt{2k}$ so that we can deduce  from Theorem~\ref{thm:4} the following risk bound for the TV-estimator.
%\jl{I have to check the bound on $\rho_T$ again, somehow, I just got a different value...}

\begin{proposition} \label{prop:4.2}
Assume that we observe the random vector $\by = \boldf^*+\bxi$, where $\boldf^*\in\mr^n$ is the fixed but unknown vector of interest obscured by Gaussian noise $\bxi\sim \sigma^* \mathcal N(0,\bfI_n)$. Let $\delta>0$ be a constant and $k$ be the smallest integer larger than $\left(\|\boldf^\uparrow\|_{\rm TV}^2n\log(n/\delta)/\sigma^*{}^2\right)^{1/3}$, where $\boldf^\uparrow$ is the orthogonal projection of $\boldf^*$ on the convex polyhedral cone of vectors with nondecreasing entries. Then, for the tuning parameter $\lambda = \sigma^* \sqrt{\log(n/\delta)/({kn})}$, the TV-estimator (\ref{eq:TV}) fulfills
\begin{align*}%\label{eq:4.4}
\frac1n\|\hat\boldf{}^{\rm TV}-\boldf^*\|_2^2 \le
\frac1n\|\boldf^{\uparrow}-\boldf^*\|_2^2 +
\frac{2\sigma^*{}^2(1+2\log(1/\delta))}{n}+6\left(\frac{\sigma^*{}^{4}\|\boldf^\uparrow\|^2_{\rm TV}\log(n/\delta)}{n^{2}}\right)^{1/3}
\end{align*}
with probability at least $1-2\delta$.
\end{proposition}
%\jl{Reformulated proposition.}
% \begin{proposition} \label{prop:4.2}
% Let $\boldf^*\in\mr^n$ be an unknown vector and assume that we observe its noisy version $\by = \boldf^*+\bxi$, where
% $\bxi\sim \sigma^* \mathcal N(0,\bfI_n)$. Let $\boldf^\uparrow$ be the projection of $\boldf^*$, in the sense of the
% Euclidean norm of $\mr^n$, on the convex polyhedral cone of vectors with nondecreasing entries. Let $\delta\in(0,1/2)$ be
% a tolerance level and $k$ be the largest integer that does not exceed $(\|\boldf^\uparrow\|_{\rm TV}/\sigma^*)^{2/3}
% (n\log(n/\delta))^{1/3}$.
% Then, with probability at least $1-2\delta$, the TV-estimator (\ref{eq:TV}) with the tuning parameter
% $\lambda = \sigma^* (\frac{\log(n/\delta)}{kn})^{1/2}$ satisfies the inequality
% \begin{align}\label{eq:4.4}
% \frac1n\|\hat\boldf{}^{\rm TV}-\boldf^*\|_2^2 \le
% \frac1n\|\boldf^{\uparrow}-\boldf^*\|_2^2 +
% \frac{2\sigma^*{}^2(1+2\log(1/\delta))}{n}+6\Big(\frac{\sigma^*{}^2\|\boldf^\uparrow\|_{\rm TV}\log^{1/2}(n/\delta)}{n}\Big)^{2/3}.
% \end{align}
% \end{proposition}

Proposition~\ref{prop:4.2} has four crucial features. First, it provides nearly minimax rates for the TV-estimator: The dominating term is typically the last term, which is minimax up to the logarithmic factor~\cite[Eq. (1.4)]{Zhang02}. We conjecture that such logarithmic factors are always required in bounds that hold with high probability (note that the bounds in~\cite{Zhang02}, in contrast,  are in expectation). Second, the unknown quantities in the tuning parameter of Proposition~\ref{prop:4.2} can be avoided or readily estimated: The noise variance $\sigma^*$ can most likely be avoided using scaled versions of the Lasso~\cite{Belloni13,SunZhang}; the factor $\|\boldf^\uparrow\|_{\rm TV}$ measuring the total variation of the function~$\boldf^\uparrow$  can be roughly estimated\footnote{One may also expect that choosing $\lambda$ by cross validation or by minimizing an
unbiased estimator of the risk would lead to bounds similar to that of Proposition~\ref{prop:4.2}. However, there are no theoretical result corroborating this claim.} by $\max_{i,j} (y_i-y_j)$. Third, Proposition~\ref{prop:4.2} allows for model mis-specifications with respect to monotone functions. Finally, it is nonasymptotic holding for any sample size $n\geq 3$.%\\
%\jl{Could we do some Cross-validation schemes as well?}
%\jl{Reformulated paragraphs and added a few words.}

\begin{comment}
and that for every $\bu\in\mr^k$ and for $\bbeta=\bfX_T\bu\in\mr^n$, $\|\bu\|_1=|\beta_1|+|\beta_{h+1}-\beta_1|+\ldots+|\beta_{kh+1}
-\beta_{(k-1)h+1}|\le 2\|\bbeta_T\|_1\le 2k^{1/2}\, \|\bbeta_T\|_2$. On the other hand, since $\bbeta=\bfX_T\bu$, it is piecewise
constant on the partition $\mathcal T$ and $\|\bfX_T\bu\|_2^2=\|\bbeta\|_2^2\ge h\|\bbeta_T\|_2^2$.
This implies that
$$
\frac{|T|\cdot\|\bfX_T\bu\|_2^2}{n\|\bu\|_1^2} \ge \frac{kh\|\bbeta_T\|_2^2}{n\times 4k\, \|\bbeta_T\|_2^2}
=\frac{h}{4n}\ge \frac{n}{4(k+2)}\ge \frac{n}{8k}.
$$
This inequality being true for every $\bu\not =0$, we have $\nu^2_T\ge \frac{n}{8k}$.
\end{comment}
\end{example}

\begin{example}[TV-penalty for H\"older continuous functions]
In this example, we derive an almost minimax risk bound that is particularly interesting for H\"older continuous signals.  For this, we again apply Theorem~\ref{thm:4} exploiting that the TV-estimator can be considered as a special case of the Lasso. It is known that the least squares estimator with an $\ell_0$-norm penalty of the increments can achieve the minimax rate $n^{-2\alpha/(2\alpha+1)}$ up to logarithmic factors over the set of H\"older continuous functions
$\{f:[0,1]\to \mr: |f(x)-f(x')|\le L|x-x'|^\alpha\}$ with parameters $\alpha\in(0,1]$ and $L>0$~\cite{Munk}. In contrast, the best known rate for the TV-estimator over these sets is $n^{-2\alpha/3}$ and therefore clearly suboptimal~\cite{MammenvdG}. Using Theorem~\ref{thm:4}, we can improve on this bound and demonstrate  that the TV-estimator can also achieve the minimax rate  $n^{-2\alpha/(2\alpha+1)}$ up to logarithmic factors over these sets  if the tuning parameter is appropriately chosen.

\begin{proposition}\label{prop:4.3} Assume that we observe the random vector $\by = \boldf^*+\bxi$, where $\boldf^*\in\mr^n$ is the fixed but unknown vector of interest obscured by Gaussian noise $\bxi\sim \sigma^* \mathcal N(0,\bfI_n)$. Let $\delta,L>0$ and $\alpha\in(0,1]$ be constants and let
$k$ be the smallest integer larger than $\big(L^2n/(\sigma^*{}^2\log(n/\delta))\big)^{1/(2\alpha+1)}$. Moreover, let
$\mathcal H^n_{\alpha,L}=\{\boldf\in\mr^n:|f_i-f_j|\le Ln^{-\alpha}|i-j|^\alpha~\forall i,j\in[n]\}$ be the H\"older class with parameters $\alpha$ and $L$. Then, for the tuning parameter $\lambda = \sigma^* \sqrt{\log(n/\delta)/(kn)}$, the TV-estimator (\ref{eq:TV}) fulfills
\begin{align*}%\label{eq:4.7}
\frac1n\|\hat\boldf{}^{\rm TV}-\boldf^*\|_2^2
&\le \inf_{\boldf\in\mathcal H^n_{\alpha,L}}\left\{
\frac1n\|\boldf-\boldf^*\|_2^2\right\}  + \frac{8\sigma^*{}^2\log(n/\delta)}{n}
+16L^2\Big(\frac{\sigma^*{}^2\log(n/\delta)}{nL^2}\Big)^{2\alpha/(2\alpha+1)}
\end{align*}
with probability at least $1-2\delta$.
\end{proposition}
%\jl{Reformulated proposition. (I have also tried to make all displays nice and commensurate with each other - I hope that I did not add mistakes...)}
% \begin{proposition}\label{prop:4.3} Let $\by$ be as in Proposition~\ref{prop:4.2} and
% $\mathcal H^n_{\alpha,L}=\{\boldf\in\mr^n:|f_i-f_j|\le Ln^{-\alpha}|i-j|^\alpha,\forall i,j\in[n]\}$ for every $\alpha\in(0,1]$ and $L>0$.
% Let $\delta\in(0,1/2)$ be
% a tolerance level and $k$ be the largest integer that does not exceed $(\frac{L^2n}{\sigma^*{}^2\log(n/\delta)})^{1/(2\alpha+1)}$.
% Then, with probability at least $1-2\delta$, the TV-estimator (\ref{eq:TV}) with the tuning parameter
% $\lambda = \sigma^* (\frac{\log(n/\delta)}{kn})^{1/2}$ satisfies the inequality
% \begin{align*}%\label{eq:4.7}
% \frac1n\|\hat\boldf{}^{\rm TV}-\boldf^*\|_2^2
% &\le \min_{\boldf\in\mathcal H^n_{\alpha,L}}
% \frac1n\|\boldf-\boldf^*\|_2^2  + \frac{8\sigma^*{}^2\log(n/\delta)}{n}
% +16L^{2}\Big(\frac{L^2n}{\sigma^*{}^2\log(n/\delta)}\Big)^{-2\alpha/(2\alpha+1)}.
% \end{align*}
% \end{proposition}

Proposition~\ref{prop:4.3} for the TV-estimator and the risk bounds in~\cite{Munk} for the $\ell_0$-penalized least-squares estimator provide exactly the same, nearly minimax rates $(n/\log(n))^{-2\alpha/(2\alpha+1)}$. The results differ, however, in other important aspects. Benefits of Proposition~\ref{prop:4.3}, on the one hand, are its finite sample bounds and the inclusion of model mis-specifications; the risk bounds in~\cite{Munk}, in contrast, are purely asymptotic and do not take model mis-specifications into account. A deficiency of Proposition~\ref{prop:4.3}, on the other hand, is the dependence of the tuning parameter on the constants $\alpha$ and $L$.
%\jl{Reformulated this paragraph. Included model missp.}
\end{example}

%\begin{remark}
%It should be noted that in the example of total variation penalty considered above, using as $T$ a regular grid on $[n]$
%with spacing $h$ leads to $\nu^2_T\ge \frac{n}{8k}$. Therefore, if we apply Proposition \ref{prop:4.3} we will get suboptimal
%rates both for monotone functions and for H\"older continuous ones.
%\end{remark}
%\jl{I do not completely understand this remark, yet.} 
\section{Discussion}\label{sec:discussion}
\subsection{Conditions of Belloni, Chernozhukov, and Wang \cite{Belloni13}}

Using an intelligent trick, the authors of \cite{Belloni13} managed to replace $\kappa_{T,\bar c}$ by $\kappa_{T,1}$
by means of introducing a new constant $\varrho_{T,\gamma}$ which is the $1-\delta$ quantile of the following stochastic term
$\max_{\bu\in\mathcal C_0(T,\gamma)}|(\bxi/\sigma^*)^\top\bfX\bu|/\|\bfX\bu\|_2$. The fact of being able to
replace $\bar c$ by $1$ leads to a qualitative enlargement of the set of matrices satisfying the condition $\kappa_{T,\bar c}>0$.
In fact, as proved in  \cite{Belloni13}, while $\kappa_{T,1}$  is invariant by including identical columns in $\bfX$,
$\kappa_{T,\bar c}$ with any $\bar c>1$ vanishes if we copy a column of $\bfX_T$ in $\bfX_{T^c}$ or vice-versa.
Note that this property of invariance by copying covariates from $T$ to $T^c$ and, reciprocally, from $T^c$ to $T$
holds true for the weighted compatibility factors defined in Section~\ref{sec:fast} as well.

Combining our approach with the idea of \cite{Belloni13}, it is possible to replace  $\bar\kappa_{T, \gamma, \bomega}^{-1}$
in (\ref{eq:3.1}) by $\frac1{|T|}\bar\varrho_{T,\gamma,\bomega}^2+\kappa_{T, 1}^{-1}$, where $\bar\varrho_{T,\gamma,\bomega}$ is
the $(1-\delta)$-quantile of the random variable $\max_{\bu\in\mathcal C_0(T,\gamma,\bomega)}|\bxi^\top(\bfI_n-\Pi_T)\bfX_{T^c}\bu_{T^c}|/\|\bfX\bu\|_2$. At first sight, this upper bound is tighter
than the one of Theorem~\ref{thm3}, but it is less interpretable because of the presence of $\bar\varrho_{T,\gamma,\bomega}$.
However, as we prove below in Proposition \ref{prop:3}, quantities like $\bar\varrho_{T,\gamma,\bomega}$ do not really lead to a
substantially smaller risk bound than the one expressed in terms of compatibility factors. To complete the
comparison of our results with those in \cite{Belloni13}, let us simply remark that since our model is simpler than the one
considered in  \cite{Belloni13}, the results we get are sharper. Indeed, we get a leading constant one in the oracle inequality,
while the proof technique of \cite{Belloni13} would produce a leading constant strictly larger than one in the mis-specified
case.

\begin{proposition}\label{prop:3}
Let $\delta\in(0,1/2)$ and $\varrho_{T,\bar c}$ be the $1-\delta$ quantile of the random variable
$\bar\eta=\sup_{\bu\in\mr^p:\|\bu_{T^c}\|_1\le \bar c\|\bu_T\|_1}
\frac{|\bxi^\top\bfX\bu|}{\sigma^*\|\bfX\bu\|_2}$.
For every $ 10\le |J|\le n$ we have
\begin{equation}
\varrho_{T,\bar c} \ge \frac{\lambda_{\min,J}|T|^{1/2}}{16\kappa_{T,\bar c}^{1/2}}\bigwedge \frac{|J|^{1/2}}{4},
\end{equation}
where $\lambda_{\min, J}$ is the smallest singular value of the matrix $\frac1{\sqrt{n}}\bfX_J$.
\end{proposition}

In many concrete examples of design matrices $\bfX$, there is a set $J\subset T^c$ of cardinality of the same order as $n$ such that
$\bfX_J$ is of full rank. For such matrices,  $\lambda_{\min, J}$ is a constant and the proposition tells us that $|T|^{-1}\varrho_{T,\bar c}^2$
is of the order of $\kappa_{T,\bar c}^{-1}\wedge (n/|T|)$. As we already mentioned, the risk bound in \cite{Belloni13} is proportional to
$|T|^{-1}\varrho_{T,\bar c}^2+\kappa_{T,1}^{-1}\asymp \kappa_{T,\bar c}^{-1}$. Therefore, according to the result of the last proposition,
there is no significant gain in the rate of convergence nor in the severity of the assumptions imposed on $\bfX$ when using
$\varrho_{T,\bar c}$ instead of $\kappa_{T,\bar c}$.

\begin{remark}
The constant $\bar\varrho_{T,\bar c}$ slightly differs from the one used in \cite{Belloni13}, where the additional constraint
$\|\betaz+\bu\|_1\le \bar c\|\betaz\|_1$ is included in the definition of $\mathcal C_0(T,\bar c)$.
On the one hand, more generally, one can define the set $\mathcal C_0(T,\bar c,\bbeta)=\{\bu\in\mr^p:\|\bbeta+\bu\|_1- \|\bbeta_T\|_1<(\bar c-1)(\bar c+1)^{-1}\|\bu\|_1\}$, contained in the one used in \cite{Belloni13}, and let
$\varrho_{T,\bar c,\bbeta}$ be the $1-\delta$ quantile of
$$
\bar\eta=\sup_{\bu\in\mathcal C_0(T,\bar c,\bbeta)}
\frac{|\bxi^\top\bfX\bu|}{\sigma^*\|\bfX\bu\|_2}.
$$
Then the risk bound of \cite{Belloni13} holds true with $\bar\varrho_{T,\bar c}$ replaced by $\bar\varrho_{T,\bar c,\betaz}$. On the
other hand, if one looks for a characteristic independent of $\betaz$, then it is necessary to take a supremum over all $\betaz$ that
are zero outside $T$. This amounts to taking the supremum over $\cup_{\betaz:\|\betaz_{T^c}\|_1=0} \mathcal C_0(T,\bar c,\betaz)=\mathcal C_0(T,\bar c)$,
that is to considering the quantity $\bar\varrho_{T,\bar c}$ of Proposition~\ref{prop:3}.
%\begin{comment}
Finally, one can repeat the arguments of the proof of Proposition~\ref{prop:3} to check that
for every $J\subset T^c$ satisfying $ 10\le |J|\le n$, we have
$\varrho_{T,\bar c,\betaz} \ge \frac{\lambda_{\min,J}|T|^{1/2}}{16\tilde\kappa_{T,\bar c,\betaz}^{1/2}}\wedge \frac{|J|^{1/2}}{4}$,
where $\tilde\kappa_{T,\bar c,\betaz}$  is the $\inf$ over all ${\bu\in\mathcal C(T,\bar c,\betaz)}$ of the
ratio $\frac1n{\|\bfX\bu\|_2^2}/{(\|\betaz_T\|_1+(\bar c-1)(\bar c+1)^{-1}\|\bu\|_1-\|\betaz+\bu\|_1)}$.
%\end{comment}
\end{remark}

\subsection{Relation to the previous work on the Lasso with correlated covariates}\label{ss:5.2}
Our bound (\ref{eq:4.2}) is close in spirit to the bound in \cite[Theorem 3.1]{HebiriLederer} which is a consequence of results in \cite{vdGeer11}.
Both bounds demonstrate that the Lasso can achieve fast rates for prediction even for highly correlated design matrices if the tuning parameter is chosen appropriately. %(Note, however, that the described choices of the tuning parameters depend in both cases on the true regression vector $\beta^*$ and are therefore not directly applicable in practice.)
In order to make the comparison with our results easier, let us state the main result of \cite{vdGeer11} using the notation of the present work.
In fact, Theorem 4.1 in \cite{vdGeer11} establishes that for any $\alpha\in [0,1)$, $\lambda>0$, and $\lambda_0>0$, on the event $\mathcal B_{\alpha,\lambda_0}=\{\sup_{\bbeta:\|\bbeta\|_1=1} 4|\bxi^\top\bfX\bbeta|/\|\bfX\bbeta\|_2^{1-\alpha}\le n^{(1+\alpha)/2}\lambda_0\}$, it holds that
\begin{align}\label{JohSara}
\ell_n(\betaL\!\!,\betaz)&\le 7\inf_{\bar\bbeta\in\mr^p,T\subset[p]} \bigg\{\ell_n(\bar\bbeta,\betaz)+\frac{8\lambda}{3}\|\bar\bbeta_{T^c}\|_1+
\frac{7}{6}\bigg(\frac{\lambda_0}{\lambda^\alpha}\bigg)^{\frac{2}{1-\alpha}}\!\!+\frac{224\lambda^2|T|}{\kappa_{T,6}}\bigg\},
\end{align}
where $\kappa_{T,6}$ is the compatibility constant (\ref{kappa1}). Furthermore, the authors of \cite{vdGeer11} provide sufficient conditions in terms of the entropy of the set
$\mathcal F=\{\bbeta\in\mr^p:\|\bfX\bbeta\|_2^2\le n; \|\bbeta\|_1\le 1\}$ ensuring that the probability of the event
$\mathcal B_{\alpha,\lambda_0}$ is close to one for some $\alpha\in(0,1)$ and $\lambda_0\asymp (\log(n)/n)^{1/2}$. The main
advantages of the results stated in the present work as compared to (\ref{JohSara}) are that (a) the risk bounds are with
leading constant one, (b) the quantity $\rho_T$ governing the choice of $\lambda$ and the rate of convergence of the
prediction risk is, in general, easier to compute than the entropy, and (c) the compatibility factor is replaced by the weighted
compatibility factor that is strictly positive in several important cases where the compatibility factor vanishes (\textit{e.g.},
for total variation penalization). Moreover (d), the benefits of our results do not necessarily rely on a high overall correlation: If, for example, the covariates can be clustered into (a reasonably small number of) sets of highly correlated covariates, one can find a small set $T$ such that the measure $\rho_T$ is small. In contrast, the entropy measure in~\cite{vdGeer11} is not necessarily small for this example, because it measures the symmetric convex hull of all variables, which is still a large set if the clusters are not highly correlated with each other (we expect, however, that the approach in~\cite{vdGeer11} can be refined in this respect). On the other hand, risk bound (\ref{JohSara}) may potentially offer more flexibility
due to the presence of the parameter $\alpha$.  Besides, it is very likely that the proof technique used in this work
allow for removing the factor $7$ in front of the $\inf$ at the right hand-side of (\ref{JohSara}).

There is another direction of research, explored in the recent paper \cite{BRGZ13} and in discussions \cite{Wegkamp13,Samworth},
that replaces the original design matrix by a new one with weaker correlations between the covariates. This is achieved by clustering
the columns of $\bfX$ and replacing the groups of strongly correlated covariates by one representer (CRL), or by gathering strongly
correlated covariates in disjoint groups for prediction with the group-Lasso procedure (CGL). While the theoretical results developed
in \cite{BRGZ13} demonstrate advantages with respect to those available for the Lasso,  the experimental results reported in Tables 2-5
show that the Lasso remains perfectly competitive with the new procedures CRL and CGL. The results of this work explain, at least
partially, these empirical results, in that we proved that the prediction loss of the standard Lasso is small even if the design matrix
contains strongly correlated columns.

Furthermore, we believe that the clustering strategy developed in \cite{BRGZ13} may be beneficial in conjunction with the Lasso, without
any modification. In fact, the result of the clustering can be used for tuning the penalty level $\lambda$. More precisely, let
$G_1,\ldots,G_M$ be the clusters we get, forming a partition of $[p]$. Theoretical results developed in previous sections suggest
to choose one representer per cluster by setting $j_m=\text{arg}\min_{i\in G_m} \max_{j\in G_m}\|(\bfI_n-\Pi_{i})\bx^j\|_2$ and
to define $T=\{j_1,\ldots,j_M\}$ together with $\lambda=2\sigma^*\rho_T \sqrt{2\log(p/\delta)/n}$. If the clusters are tight---the
vectors within each cluster are very close to one another---then $\rho_T$ will be small, since
$\rho_T=n^{-1/2}\max_j \|(\bfI_n-\Pi_T)\bx^j\|_2\le n^{-1/2}\max_m\max_{j\in G_m} \|(\bfI_n-\Pi_T)\bx^j\|_2\le
n^{-1/2}\max_m\max_{j\in G_m} \|(\bfI_n-\Pi_{j_m})\bx^j\|_2$. Note that once the clusters are found, the aforementioned
computation of $j_m$'s and of $\lambda$ is not time consuming. Another compelling alternative is to replace the clustering step by
sparse PCA. Indeed, our theoretical findings advocate for choosing as $T$ a subset of $[p]$ which is simultaneously of small
cardinality and  such that all $\bx^j$'s are close to the linear space spanned by $\{\bx^j:j\in T\}$. It is precisely the task of the
sparse PCA to find such a subset $T$ (cf.\ \cite{Birnbaum} and the references therein).

\section{Conclusions}\label{sec:conclusion}

%%%%%%%%%%%%%%%%

Our results lead to a better understanding of the prediction performance of the Lasso, as they demonstrate that correlations are not necessarily obstructive but even helpful in some cases. This permits more accurate comparisons of the Lasso with its many competitors. Our results are based on the introduction\footnote{A very similar quantity appears also in the risk bound provided by Theorem 8.2 in \cite{Koltchinskii_book}, but the result therein does not suggest the incorporation of this quantity into the tuning parameter, and the established risk bounds are less
accurate than those of the present work.} of $\rho_T$, a simple measure of the correlations of the covariates. If this measure is incorporated in the choice of the tuning parameter, the Lasso prediction risk decays at a fast rate for a broad variety of settings including settings with strongly correlated covariates. To derive this, we did not invoke the usual assumptions such as restricted isometry, restricted eigenvalues, etc., but rather relied on the slow rate bounds that hold for arbitrary designs. Consequences of our results are then substantially improved risk bounds for the least squares estimator with total variation penalty.

We also introduce compatibility factors that are not only abstract concepts but  can be both evaluated numerically and bounded theoretically, see Example~\ref{prop:3.1}. We introduce in particular a new, weighted compatibility factor $\bar\kappa$, which---in contrast to its original version---may be bounded away from zero even for strongly correlated covariates. This allows us to apply the corresponding results to the least squares estimator with total variation penalty, for example, where the correlations between the covariates are up to $1 - (1/n)$.

Our results finally indicate that the prediction performance of the Lasso can not be characterized by only the maximal correlation between covariates: On the one hand, as described above, the Lasso can provide accurate prediction even if the covariates are highly correlated. On the other hand, as indicated by Example~\ref{example1}, the Lasso can perform poorly in prediction even for moderately correlated covariates.

Future research directions include developing our approach in the case of the group Lasso \cite{Yuan06,Lounici11}
in order to understand how to optimally
exploit the (correlation) structure of the Gram matrix for defining the groups. This problem is also interesting for applications
to the total-variation penalization as discussed in \cite{Vert}. Another relevant question is how the refinements proposed in the
present work may be adapted to scale invariant versions of the Lasso, such as the square-root Lasso \cite{Belloni13}, scaled Lasso
\cite{Stadler} or scaled Dantzig selector \cite{DalalyanC12}. We also believe that the geometry of the design, measured by the quantities
$\rho_T$, may lead to better recommendations for the tuning parameter in the transductive setting \cite{AlquierMo}.
Finally, we would like to explore the consequences of our results when applied to the nonparametric estimation of a regression function
$f$ by penalized least squares with a penalty proportional to the discrete counterpart of the $L_1$-norm of the
$k$th derivative of $f$. This problem has been studied in \cite{MammenvdG}, but we believe that the results of the
present work may lead to improved risk bounds.

\section{Proofs}\label{sec:proofs}

In this section, we gather the proofs of all the theorems  and propositions stated in previous sections. For ease of notation,  we write $\betae$ instead of $\betaL$ throughout these proofs.
In the sequel, we denote by $\sign(x)$ the sub-differential of the function
$x\mapsto |x|$, that is
$$
\sign(x) = \begin{cases}
\{1\},& x>0,\\
[-1,1],& x=0,\\
\{-1\}, & x<0.
\end{cases}
$$

\begin{proof}[Proof of Theorem~\ref{thm1}]
We first use the Karush-Kuhn-Tucker conditions to infer that
\begin{align*}
\frac{1}{n}\bfX^\top(\by-\bfX\betae) \in
\lambda\sign(\betae).%\label{KKT}.
\end{align*}
This implies that for every vector $\barbbeta\in\mr^p$
\begin{align*}
\frac{1}{n}\betae_T^\top\bfX_T^\top(\by-\bfX\betae) &=\lambda\|\betae_T\|_1,\\%\label{KKT1}\\
\frac{1}{n}\barbbeta^\top_{T}\bfX_{T}^\top(\by-\bfX\betae) &\le \lambda\|\barbbeta_{T}\|_1.%\label{KKT2}.
\end{align*}
Subtracting the first relation from the second one, we get
\begin{align}\label{eq:2.1}
\frac{1}{n}(\barbbeta-\betae)_T^\top\bfX_T^\top(\by-\bfX\betae) &\le \lambda(\|\barbbeta_T\|_1-\|\betae_T\|_1).
\end{align}
We define now the vector $\barbbeta$ by the relations
$$
\barbbeta_T = \betazT_T+(\bfX_T^\top\bfX_T)^\dag\bfX_T^\top\bfX_{T^c}(\betazT-\betae)_{T^c}\qquad\text{and}\qquad \barbbeta_{T^c} = 0.
$$
This choice of $\barbbeta$ may appear somewhat strange and complicated, but is made in order that the relation
$\bfX_T(\barbbeta-\betae)_T =\Pi_T\bfX(\betazT-\betae)$ be satisfied. On the other hand, one easily checks that
$\by= \bfX_T\betaz_T+\bfX_{T^c}\betaz_{T^c}+\bxi = \bfX\betazT+(\bfI_n-\Pi_T)\bxi$.
Replacing these expressions of $\barbbeta$ and $\by$ in (\ref{eq:2.1}), we find that for every $T\subset [p]$,
\begin{align*}
\frac{1}{n}(\betazT-\betae)^\top\bfX^\top\Pi_T\bfX(\betazT-\betae) &\le \lambda(\|\barbbeta_T\|_1-\|\betae_T\|_1)\le \lambda
\|(\barbbeta-\betae)_T\|_1.
\end{align*}
Equivalently, this relation may be written as
\begin{align*}
\frac{1}{n}\|\Pi_T\bfX(\betazT-\betae)\|_2^2 = \frac{1}{n}\|\bfX_T(\barbbeta-\betae)_T\|_2^2  \le \lambda
\|(\barbbeta-\betae)_T\|_1.
\end{align*}
In view of the fact that $\sqrt{|T|/n}\,\|\bfX_T \bu\|_2\ge \nu_T\|\bu\|_1$ for every $\bu\in\mr^{|T|}$, we get
\begin{align*}
\|(\barbbeta-\betae)_T\|_1\le \frac{\sqrt{|T|}\|\bfX_T(\barbbeta-\betae)_T\|_2 }{\sqrt{n}\,\nu_T}.
\end{align*}
Combining the last two displays, we obtain
$$
\frac{1}{n}\|\Pi_T\bfX(\betazT-\betae)\|_2^2 = \frac{1}{n}\|\bfX_T(\barbbeta-\betae)_T\|_2^2  \le \lambda
\frac{\sqrt{|T|}\|\bfX_T(\barbbeta-\betae)_T\|_2 }{\sqrt{n}\,\nu_T}.
$$
Dividing both sides of the last inequality by $\|\bfX_T(\barbbeta-\betae)_T\|^2_2$, we can infer the desired result.
\end{proof}

\begin{proof}[Proof of Theorem \ref{thm3}]
Recall that according to (\ref{eq:2.1}), for every $J\subset [p]$ and $\barbbeta\in\mr^p$,
\begin{align*}
\frac{1}{n}(\barbbeta-\betae)_J^\top\bfX_J^\top(\by-\bfX\betae) &\le \lambda(\|\barbbeta_J\|_1-\|\betae_J\|_1).
\end{align*}
Replacing the expression of $\by$ in this inequality, we find that
\begin{align*}
\frac{1}{n}(\barbbeta-\betae)_J^\top\bfX_J^\top(\bfX\betaz+\bxi-\bfX\betae) &\le \lambda(\|\barbbeta_J\|_1-\|\betae_J\|_1).
\end{align*}
Let us introduce the two difference vectors $\bdelta = \betae-\betaz$ and $\bar\bdelta =  \betae-\barbbeta$. The last
display combined with the decomposition $\bxi = \Pi_T\bxi+(\bfI_n-\Pi_T)\bxi$, for every $T\subset [p]$, yields
\begin{align*}
\frac{1}{n}\bar\bdelta_J^\top\bfX_J^\top\bfX\bdelta  &\le  \frac1n \bar\bdelta_J^\top\bfX_J^\top\Pi_T\bxi
+\frac1n \bar\bdelta_J^\top\bfX_J^\top(\bfI_n-\Pi_T)\bxi+ \lambda(\|\barbbeta_J\|_1-\|\betae_J\|_1).
\end{align*}
Using the identity $\bu^\top\bu' = \frac12(\|\bu\|_2^2+\|\bu'\|_2^2-\|\bu-\bu'\|_2^2)$ we get for every $J\subset[p]$,
\begin{align}\label{1.1}
\frac{\|\bfX_J\bar\bdelta_J\|_2^2+\|\bfX\bdelta\|_2^2}{n}  &\le
\frac1n\|\bfX_J\bar\bdelta_J-\bfX\bdelta\|_2^2+ \frac2n \bar\bdelta_J^\top\bfX_J^\top(\bfI_n-\Pi_T)\bxi+\frac2n \bar\bdelta_J^\top\bfX_J^\top
\Pi_T\bxi\nonumber\\
&\qquad + 2\lambda(\|\barbbeta_J\|_1-\|\betae_J\|_1).
\end{align}
To prove (\ref{eq:3.1}) we choose $J=[p]$, for which (\ref{1.1}) becomes	
\begin{align}\label{1.2}
\frac{1}{n}\|\bfX\bar\bdelta\|_2^2+\frac1n\|\bfX\bdelta\|_2^2  &\le
\frac1n\|\bfX\bar\bdelta-\bfX\bdelta\|_2^2+ \frac2n \bar\bdelta^\top\bfX^\top\bxi+ 2\lambda(\|\barbbeta\|_1-\|\betae\|_1)\\
&\le
\frac1n\|\bfX\bar\bdelta-\bfX\bdelta\|_2^2+ \frac2n \bar\bdelta^\top_{T^c}\bfX^\top_{T^c}(\bfI_n-\Pi_T)\bxi
+\frac2n \bar\bdelta^\top\bfX^\top\Pi_T\bxi \nonumber\\
&\qquad+ 2\lambda(\|\bar\bdelta_{T}\|_1-\|\bar\bdelta_{T^c}\|_1)+4\lambda\|\barbbeta_{T^c}\|_1,\qquad\forall T\subset [p].\label{1.21}
\end{align}
Notice that
$\bar\bdelta^\top_{T^c}\bfX^\top_{T^c}(\bfI_n-\Pi_T)\bxi\le \sum_{j\in T^c}|\bx^j{}^\top(\bfI_n-\Pi_T)\bxi|\cdot|\bar\delta_j|$ and $|\bar\bdelta^\top\bfX^\top\Pi_T\bxi|\le \|\bfX\bar\bdelta\|_2\|\Pi_T\bxi\|_2$.
Replacing $\lambda$ by its value $\gamma \sigma^*\big(\frac2n\log(p/\delta)\big)^{1/2}$ and restricting our attention to the event
$\mathcal B_T$ defined as the intersection of the events $\big\{\max_{j\in T^c} |\bxi^\top(\bfI_n-\Pi_T)\bx^j|/\|(\bfI_n-\Pi_T)\bx^j\|_2\le \sigma^*\big(2\log(p/\delta)\big)^{1/2}\big\}$
and $\big\{\|\Pi_T\bxi\|_2\le \sigma^*\big(\sqrt{|T|}+\sqrt{2\log(1/\delta)}\big)\big\}$ we obtain
\begin{align}\label{1.4}
\frac{1}{n}\|\bfX\bar\bdelta\|_2^2+\frac1n\|\bfX\bdelta\|_2^2   &\le
\frac1n\|\bfX\bar\bdelta-\bfX\bdelta\|_2^2+ 4\lambda\|\barbbeta_{T^c}\|_1+\frac2n\|\bfX\bar\bdelta\|_2\|\Pi_T\bxi\|_2 \nonumber\\
&\qquad+{2\lambda} \Big(\|\bar\bdelta_{T}\|_1-\|\bar\bdelta_{T^c}\|_1+\gamma^{-1}\|(\bar\bdelta\odot\bomega)_{T^c}\|_1\Big).
\end{align}
The definition of $\bar\kappa_{T,\gamma,\bomega}$ implies that
$\|\bar\bdelta_{T}\|_1-\|\bar\bdelta_{T^c}\|_1+\gamma^{-1}\|(\bar\bdelta\odot\bomega)_{T^c}\|_1 \le
\frac{|T|^{1/2}\cdot \|\bfX\bar\bdelta\|_2}{(n\bar\kappa_{T,\gamma,\bomega})^{1/2}}$ (note that this inequality is
trivial when the left hand-side is negative), therefore
\begin{align}\label{1.41}
\frac2n\|\bfX\bar\bdelta\|_2\|\Pi_T\bxi\|_2&+{2\lambda} \Big(\|\bar\bdelta_{T}\|_1-\|\bar\bdelta_{T^c}\|_1+\gamma^{-1}\|(\bar\bdelta\odot\bomega)_{T^c}\|_1\Big)
\nonumber\\
&\le 2\frac{\|\bfX\bar\bdelta\|_2}{\sqrt{n}}\Big(\frac1{\sqrt{n}}\|\Pi_T\bxi\|_2+\lambda (|T|/\bar\kappa_{T,\gamma,\bomega})^{1/2}\Big) \nonumber\\
&\le \frac{\|\bfX\bar\bdelta\|_2^2}{n}+\Big(\frac1{\sqrt{n}}\|\Pi_T\bxi\|_2+\lambda (|T|/\bar\kappa_{T,\gamma,\bomega})^{1/2}\Big)^2\nonumber\\
&\le \frac{\|\bfX\bar\bdelta\|_2^2}{n}+\frac{4\sigma^*{}^2(|T|+2\log(1/\delta))}{n}+\frac{4\gamma^2\sigma^*{}^2|T|\log(p/\delta)}
{n\bar\kappa_{T,\gamma,\bomega}}.
\end{align}
After replacing (\ref{1.41}) in (\ref{1.4}), we remark that the terms $\frac{1}{n}\|\bfX\bar\bdelta\|_2^2$ cancel out and we get inequality (\ref{eq:3.1}).
Classical results on the tails of Gaussian and $\chi^2$ distributions imply that $\mathcal B_T$ is at least of probability $1-2\delta$, for every $T$.
The assertion of (\ref{eq:3.1}) follows by choosing $T$ to be a subset of $[p]$ minimizing the right hand-side of (\ref{eq:3.1}).

%To check (\ref{eq:3.2}), we start by applying (\ref{1.4}) with $\bar\bbeta = \betaz$ and $T=J^*$. This yields
%\begin{align}\label{1.10}
%\frac1n\|\bfX\bdelta\|_2^2  &\le \frac{1}{n}\|\bfX\bdelta\|_2\|\Pi_T\bxi\|_2+{\lambda} \Big(\|\bdelta_{J^*}\|_1-\|\bdelta_{J^*{}^c}\|_1+\gamma^{-1}\|(\bdelta\odot\bomega)_{J^*{}^c}\|_1\Big).
%\end{align}
%On the other hand, relation (\ref{KKT1}) implies that for every $J\subset [p]$
%\begin{align*}
%\frac{1}{n}\betae_J^\top\bfX_J^\top(\bfX\bdelta-\bxi) &= -\lambda\|\betae_J\|_1.
%\end{align*}
%Choosing $J=J^*{}^c$ and using the fact that $\bdelta_{J^*{}^c}=\betae_{J^*{}^c}$, we get
%\begin{align*}
%\frac{1}{n}\bdelta_{J^*{}^c}^\top\bfX_{J^*{}^c}^\top\bfX\bdelta
%	&=\frac{1}{n}\bdelta_{J^*{}^c}^\top\bfX_{J^*{}^c}^\top\bxi -\lambda\|\bdelta_{J^*{}^c}\|_1
%	\le \lambda\gamma^{-1} \cdot \|\bdelta_{J^*{}^c}\|_1-\lambda\|\bdelta_{J^*{}^c}\|_1\le 0.
%\end{align*}
%Hence, the right hand-side of (\ref{1.10}) is upper bounded by the quantity
%$\frac{1}{n}\|\bfX\bdelta\|_2\|\Pi_T\bxi\|_2+\lambda \|\bfX\bdelta\|_2/(n\bar\kappa^{(1)}_{T,\gamma,\bomega})^{1/2}$. Dividing both sides of the
%resulting inequality by $\|\bfX\bdelta\|_2$ and using that on $\mathcal B_T$ we have $\|\Pi_T\bxi\|_2\le \sqrt{|T|}+ \sqrt{2\log(1/\delta)}$,
%this yields the second assertion of the theorem.
The proof of the second claim of the theorem is identical to that of  (\ref{eq:3.1}) and,
therefore, is left to the reader.
\end{proof}

\begin{proof}[Proof of Theorem~\ref{thm:4}]
According to (\ref{1.2}), for every $\bar\bbeta\in\mr^p$ and for $\bdelta = \betae-\betaz$ and $\bar\bdelta = \betae-\bar\bbeta$, the inequality
\begin{align}\label{2.8}
\frac{1}{n}\|\bfX\bar\bdelta\|_2^2+\frac1n\|\bfX\bdelta\|_2^2  &\le
\frac1n\|\bfX\bar\bdelta-\bfX\bdelta\|_2^2+ \frac2n \bar\bdelta^\top\bfX^\top\bxi+ 2\lambda(\|\barbbeta\|_1-\|\betae\|_1)
\end{align}
holds for every $\lambda>0$. We split the stochastic term $\bar\bdelta^\top\bfX^\top\bxi$ into two terms
$\bar\bdelta^\top\bfX^\top(\bfI_n-\Pi_T)\bxi$ and $\bar\bdelta^\top\bfX^\top\Pi_T\bxi$. The first one can be bounded using the
duality inequality $\bar\bdelta^\top\bfX^\top(\bfI_n-\Pi_T)\bxi=\bar\bdelta_{T^c}^\top\bfX_{T^c}^\top(\bfI_n-\Pi_T)\bxi\le
\|\bar\bdelta_{T^c}\|_1\|\bfX_{T^c}^\top(\bfI_n-\Pi_T)\bxi\|_\infty\le (\|\barbbeta_{T^c}\|_1+\|\betae_{T^c}\|_1)\|\bfX_{T^c}^\top(\bfI_n-\Pi_T)\bxi\|_\infty$, while the second one, in view of the Cauchy-Schwarz inequality,
satisfies $\bar\bdelta^\top\bfX^\top\Pi_T\bxi\le \|\bfX\bar\bdelta\|_2\|\Pi_T\bxi\|_2\le  (\|\bfX\bar\bdelta\|^2_2+\|\Pi_T\bxi\|^2_2)/2$.
As in the proof of fast rates, we restrict our attention to the event
$\mathcal B_T$ defined as the intersection of the events $\big\{\max_{j\in T^c} |\bxi^\top(\bfI_n-\Pi_T)\bx^j|/\|(\bfI_n-\Pi_T)\bx^j\|_2\le \sigma^*\big(2\log(p/\delta)\big)^{1/2}\big\}$
and $\big\{\|\Pi_T\bxi\|_2\le \sigma^*\big(\sqrt{|T|}+\sqrt{2\log(1/\delta)}\big)\big\}$. On this event, we get
\begin{align}\label{2.9}
2\bar\bdelta^\top\bfX^\top\bxi\le \|\bfX\bar\bdelta\|^2_2+2\sigma^*{}^2(|T|+2\log(1/\delta)) + 2 \lambda\gamma^{-1} n (\|\barbbeta_{T^c}\|_1+\|\betae_{T^c}\|_1).
\end{align}
Combining this inequality with (\ref{2.8}), we get the desired result.
\end{proof}

\newcommand{\ej}{\ensuremath{\betae}}
\newcommand{\ejp}{\ensuremath{\betae'}}
\newcommand{\beu}{\ensuremath{\ej^u}}
\newcommand{\obj}{\ensuremath{g}}

\begin{proof}[Proof of Proposition \ref{prop:2}]
We first recall that given data $(\by,\bfX)$ and tuning parameter $\lambda,$ all solutions~\eqref{lasso} provide the same prediction, which means that it is sufficient to consider only one of the solutions. For completeness, let us briefly derive this fact. We do this by contradiction, considering two solutions $\ej,\ejp$
\begin{equation*}
\ej,\ejp\in\text{arg}\min_{\bbeta}\obj(\bbeta),~~~~\obj(\bbeta):=\frac{1}{2n}\|\by-\bfX\bbeta\|_2^2+\lambda\|\bbeta\|_1,
\end{equation*}
and assuming that these solutions lead to different predictions, that is, $\bfX\ej\neq\bfX\ejp.$ Fix now a number~$u\in(0,1)$ and introduce the linear combination $\beu:=u\ej+(1-u)\ejp.$  The definition of $\ej,\ejp$ as minimizers of the objective function~$\obj$ implies $\obj(\ej)=\obj(\ejp)\leq \obj(\beu).$ On the other hand, we can deduce the strict inequality $\obj(\beu)<u\obj(\ej)+(1-u)\obj(\ejp)$ from the strict convexity of the objective function~$g,$ which in turn is due to the strict convexity of the function $\boldsymbol\gamma\mapsto\|\by-\boldsymbol\gamma\|_2^2.$  Combining these findings leads to $\obj(\beu)<u\obj(\beu)+(1-u)\obj(\beu)=\obj(\beu)$. This is a contradiction, so that we can conclude that for given $(\by,\bfX)$ and $\lambda,$ all Lasso solutions provide the same prediction.\\
% Therefore, we have the strict inequality
% \begin{align*}
% f(\beu)<&uf(\betae_\lambda)+(1-u)f(\betae'_\lambda).
% \end{align*}
% \begin{align*}
% \frac{1}{2n}\|\by-\bfX\betae_\lambda(u)\|_2^2+\lambda\|\betae_\lambda(u)\|_1<&u\left(\frac{1}{2n}\|\by-\bfX\betae_\lambda\|_2^2+\lambda\|\betae_\lambda\|_1\right)\\
% &~~~~+(1-u)\left(\frac{1}{2n}\|\by-\bfX\betae'_\lambda\|_2^2+\lambda\|\betae'_\lambda\|_1\right).
% \end{align*}
% On the other hand,  With the previous display, this gives
% \begin{equation*}
% f(\beu)<uf(\beu)+(1-u)f(\beu)=f(\beu).
% \end{equation*}
% \begin{align*}
% \frac{1}{2n}\|\by-\bfX\betae_\lambda\|_2^2+\lambda\|\betae_\lambda(u)\|_1<&u\left(\frac{1}{2n}\|\by-\bfX\betae_\lambda\|_2^2+\lambda\|\betae_\lambda\|_1\right)\\
% &~~~~+(1-u)\left(\frac{1}{2n}\|\by-\bfX\betae_\lambda\|_2^2+\lambda\|\betae_\lambda\|_1\right)=\frac{1}{2n}\|\by-\bfX\betae_\lambda\|_2^2+\lambda\|\betae_\lambda\|_1.
% \end{align*}
% \begin{equation*}
% \frac{1}{2n}\|\by-\bfX\betae_\lambda(u)\|_2^2+\lambda\|\betae_\lambda(u)\|_1<\frac{1}{2n}\|\by-\bfX\betae_\lambda\|_2^2+\lambda\|\betae_\lambda\|_1.
% \end{equation*}
% This is a contradiction, so that we can conclude that all Lasso solutions provide the same prediction.\\
We can now look at specific solutions to~\eqref{lasso} for Example~\ref{example1}. To this end, we note that the Karush-Kuhn-Tucker conditions for Example~\ref{example1} are the following:
\begin{align*}
\frac1{\sqrt2}(\be_1+\be_{j+1})^\top(\sqrt{2}\be_1+\frac1{\sqrt{n}}\bxi-
\frac1{\sqrt2}\sum_{k=1}^{2m}\hat\beta_k\be_1 -\frac1{\sqrt2}(\hat\beta_j-\hat\beta_{m+j})\be_{j+1} ) &\in \lambda\sign(\widehat\beta_j),\\
\frac1{\sqrt2}(\be_1-\be_{j+1})^\top(\sqrt{2}\be_1+\frac1{\sqrt{n}}\bxi-
\frac1{\sqrt2}\sum_{k=1}^{2m}\hat\beta_k\be_1 -\frac1{\sqrt2}(\hat\beta_j-\hat\beta_{m+j})\be_{j+1} ) &\in \lambda\sign(\widehat\beta_{m+j}).
\end{align*}
After simplification, we get
\begin{align}
1+\frac1{\sqrt{2n}}(\xi_1+\xi_{j+1})-
\frac12\sum_{k=1}^{2m}\hat\beta_k -\frac1{2}(\hat\beta_j-\hat\beta_{m+j}) &\in \lambda\sign(\widehat\beta_j),\label{eq:2.7}\\
1+\frac1{\sqrt{2n}}(\xi_1-\xi_{j+1})-
\frac12\sum_{k=1}^{2m}\hat\beta_k +\frac1{2}(\hat\beta_j-\hat\beta_{m+j}) &\in \lambda\sign(\widehat\beta_{m+j}).\label{eq:2.8}
\end{align}
We will restrict our attention to the event $\mathcal B = \{\xi_1<0\}$ which has a probability $1/2$. First note that
if $\lambda\ge 1$, then the vector $\betae = 0$ is a solution to the system (\ref{eq:2.7})-(\ref{eq:2.8}). Therefore,
$\betae = 0$ and hence $\ell_n(\betae,\betaz) = 2$. Thus, in the case
$\lambda\ge 1$ the Lasso is not consistent, which is not a surprise since the theory recommends to always choose $\lambda$
of order $O((\log(p)/n)^{1/2})$.

In the more interesting case $\lambda\in(\frac{m+1-\sqrt{2n}}{m\sqrt{2n}},1)$, a Lasso solution is given by
\begin{equation}
\hat\beta_j =
\begin{cases}
\frac{2(1-\lambda)}{m+1}, & \xi_{j+1}>0,\\
0 , & \xi_{j+1}<0,
\end{cases}\qquad
\hat\beta_{m+j} =
\begin{cases}
0 , & \xi_{j+1}>0,\\
\frac{2(1-\lambda)}{m+1}, & \xi_{j+1}<0.
\end{cases}\label{ex:lasso}
\end{equation}
Indeed, for instance if $\xi_{j+1}>0$, replacing these values of $\betae$ in (\ref{eq:2.7}) and (\ref{eq:2.8}) we get on the event $\mathcal B$:
\begin{align*}
1+\frac{\xi_1+\xi_{j+1}}{\sqrt{2n}}-\frac12\sum_{k=1}^{2m}\hat\beta_k -\frac1{2}(\hat\beta_j-\hat\beta_{m+j})
&= 1-\frac{m}2\times \frac{2(1-\lambda)}{m+1}-\frac1{2}\times\frac{2(1-\lambda)}{m+1}=\lambda\\
1+\frac{\xi_1-\xi_{j+1}}{\sqrt{2n}}-\frac12\sum_{k=1}^{2m}\hat\beta_k +\frac1{2}(\hat\beta_j-\hat\beta_{m+j})
&= 1-\frac2{\sqrt{2n}}-\frac{m}2\times \frac{2(1-\lambda)}{m+1}+\frac1{2}\times\frac{2(1-\lambda)}{m+1}\\
&=\lambda\big(1-\frac{2}{m+1}\big)-\frac{\sqrt{2}}{\sqrt{n}}+\frac{2}{m+1}\in [-\lambda,\lambda],
\end{align*}
where the last inclusion follows from the relation $m+1\ge \sqrt{2n}$. For the vector (\ref{ex:lasso}), we check that the prediction loss
\begin{align*}
\ell_n(\betae,\betaz)
	&= \Big(\sqrt{2}-\frac{\sqrt{2}m(1-\lambda)}{m+1}\Big)^2+\frac{2m(1-\lambda)^2}{(m+1)^2}\\
	&= 2\Big(\frac{1+m\lambda}{m+1}\Big)^2+\frac{2m(1-\lambda)^2}{(m+1)^2}\\
	&= \frac{2+2m\lambda^2}{m+1}	\ge  \frac{1}{m} \ge \frac{1}{\sqrt{2n}}.
\end{align*}
Finally, in the case $\lambda\in[0,\frac{m+1-\sqrt{2n}}{m\sqrt{2n}}]$, a Lasso solution on the event $\mathcal B$ is given by
\begin{equation}
\hat\beta_j =
\begin{cases}
\frac{\sqrt{2n}+m-1}{m\sqrt{2n}}-\lambda, & \xi_{j+1}>0,\\
\frac{\sqrt{2n}-m-1}{m\sqrt{2n}}+\lambda , & \xi_{j+1}<0,
\end{cases}\qquad
\hat\beta_{m+j} =
\begin{cases}
\frac{\sqrt{2n}-m-1}{m\sqrt{2n}}+\lambda, & \xi_{j+1}>0,\\
\frac{\sqrt{2n}+m-1}{m\sqrt{2n}}-\lambda, & \xi_{j+1}<0,
\end{cases}\label{ex:lasso2}
\end{equation}
for every $j\in[m]$. Indeed, for instance if $\xi_{j+1}>0$, replacing these
values of $\betae$ in (\ref{eq:2.7}) and (\ref{eq:2.8}) we get on the event $\mathcal B$:
\begin{align*}
1+\frac{\xi_1+\xi_{j+1}}{\sqrt{2n}}-\frac12\sum_{k=1}^{2m}\hat\beta_k -\frac1{2}(\hat\beta_j-\hat\beta_{m+j})
&= 1-\frac{m}2\times \frac{2(\sqrt{2n}-1)}{m\sqrt{2n}}-\frac1{2}\times\Big(\frac{2m}{m\sqrt{2n}}-2\lambda\Big)=\lambda\\
1+\frac{\xi_1-\xi_{j+1}}{\sqrt{2n}}-\frac12\sum_{k=1}^{2m}\hat\beta_k +\frac1{2}(\hat\beta_j-\hat\beta_{m+j})
&= 1-\frac{2}{\sqrt{2n}}-\frac{(\sqrt{2n}-1)}{\sqrt{2n}}+\Big(\frac{1}{\sqrt{2n}}-\lambda\Big)=-\lambda.
\end{align*}
The prediction loss of this estimator is
\begin{align*}
\ell_n(\betae,\betaz)
	&\ge m \Big(\frac{1}{\sqrt{n}}-\sqrt{2}\lambda\Big)^2\\
	&\ge 2m \Big(\frac{1}{\sqrt{2n}}-\frac{m+1-\sqrt{2n}}{m\sqrt{2n}}\Big)^2\\
	&= 2m \Big(\frac{\sqrt{2n}-1}{m\sqrt{2n}}\Big)^2\ge \frac1{2\sqrt{2n}},
\end{align*}
where for the last inequality we have used the facts that $\sqrt{2n}-1\ge \sqrt{n/2}$, $\forall n\ge 2$, and
$m\le \sqrt{2n}$. This completes the proof of the proposition.
\end{proof}

\begin{proof}[Proof of Proposition~\ref{prop:3.1}]
We will use a probabilistic argument.
Set $\boldf = \bfX\bu$. Denoting $s_j=\sign(u_j)=\sign(f_{j}-f_{j-1})$, $j\in T$, and  $s_j=-\sign(u_j)=-\sign(f_{j}-f_{j-1})$,
$j\in T^c$, with the convention that $f_0=0$, $a_{n+1}=a_n$ and $s_{n+1}=-\sign(f_n)$, we get
\begin{align*}
\|\bu_T\odot \ba_{T}\|_1-\|\bu_{T^c}\odot \ba_{T^c}\|_1 &=
\sum_{j\in T}a_j|f_{j}-f_{j-1}|-\sum_{j\in T^c}a_j|f_j-f_{j-1}|\\
&\le
\sum_{j\in T}a_j|f_{j}-f_{j-1}|-\sum_{j\in T^c}a_j|f_j-f_{j-1}|+a_n|f_n|\\
&=\sum_{j\in [n]}s_ja_j(f_{j}-f_{j-1})-a_ns_{n+1}f_n=\sum_{j\in [n]}f_j(s_ja_j-s_{j+1}a_{j+1}).
\end{align*}
More interestingly, since $|x|=\max_{t\in\{\pm1\}} tx$ for every $x\in\mr$, one easily checks
that\footnote{For notational convenience, we assume hereafter that $T$ is augmented
by the elements $\{1,n+1\}$.}
\begin{align*}
\|\bu_T\odot \ba_{T}\|_1-\|\bu_{T^c}\odot \ba_{T^c}\|_1
&=\min_{\substack{\bt\in\{\pm1\}^{n+1}\\ \bt_T=\bs_T}}\sum_{j\in [n]}f_j(t_ja_j-t_{j+1}a_{j+1}).
\end{align*}
Let us set $j_0=1$, $j_{s+1}=n+1$ and consider the partition $\{B_\ell = \llbracket j_{\ell-1},j_{\ell}\llbracket, \ell\in[s+1]\}$
of $[n]$. Let $\Delta_\ell = |B_\ell|$ for every $\ell\in[s+1]$. Permuting the minimum and the summation, we get
\begin{align*}
\|\bu_T\odot \ba_{T}\|_1-\|\bu_{T^c}\odot \ba_{T^c}\|_1
&=\sum_{\ell=1}^{s+1} \Psi_\ell
\end{align*}
where
\begin{align*}
\Psi_\ell
    &= \min_{\substack{\bt\in\{\pm1\}^{n+1}\\(t_{j_{\ell-1}},t_{j_\ell})=(s_{j_{\ell-1}},s_{j_\ell})}}
        \sum_{j=j_{\ell-1}}^{j_\ell-1} f_{j}(t_ja_j-t_{j+1}a_{j+1})\\
    &=\min_{\substack{\bar\bt\in\{\pm1\}^{\Delta_\ell+1}\\(\bar t_1,\bar t_{\Delta_\ell+1})=(s_{j_{\ell-1}},s_{j_\ell})}}
        \sum_{j=1}^{\Delta_\ell} f_{j_{\ell-1}+j-1}(\bar t_ja_{j_{\ell-1}+j-1}-\bar t_{j+1}a_{j_{\ell-1}+j}).
\end{align*}
In what follows, we propose a bound on $\Psi_1$. The other $\Psi_\ell$'s can be evaluated similarly.
Let $\eps_1,\ldots,\eps_n$ be independent random variables with values in $\pm1$ such that
$p:=\bfP(\eps_j=1)=(1-(2\Delta_1)^{-1})$ and  $\bfP(\eps_j=-1)=(2\Delta_1)^{-1}$. Further, we define
$\bar\eps_1=s_1$ and $\bar\eps_{j+1}=\bar\eps_1\times \eps_1\times\ldots\times\eps_j$ for every $j\in B_1$.
Since $\eps_j$'s are independent, $\{\bar\eps_j\}_{j\in[j_1]}$ is a Markov chain with values in $\{\pm1\}$.
For this Markov chain, we first check that
\begin{equation}
    \label{prob}
\bfP(\bar\eps_{j_1}=s_{j_1})\geq 1/4.
  \end{equation}
 Indeed, by symmetry, it suffices to consider the two cases $(s_1,s_{j_1})=(1,1)$ and $(s_1,s_{j_1})=(1,-1)$. In the first case, we
use the inclusion $\{\beps_{B_1}=\mathbf 1_{B_1}\}\subset\{\bar\eps_{j_1}=s_{j_1}\}$
to infer that
\begin{align*}
\bfP(\bar\eps_{j_1}=s_{j_1})\ge \bfP(\beps_{B_1}=\mathbf 1_{B_1})
=\prod_{j=1}^{j_1-1}\bfP(\eps_{j}=1)=(1-(2\Delta_1)^{-1})^{\Delta_1}.
\end{align*}
For $\Delta_1\ge 1$, one checks that $(1-(2\Delta_1)^{-1})^{\Delta_1}\ge 1/2$,
which yields
$\bfP(\bar\eps_{j_1}=s_{j_1})\ge \frac12$.
In the second case, $s_1=-s_{j_1}$, we use the inclusion
$\cup_{j\in B_1}\{\beps_{-j}=\mathbf 1,\eps_j=-1\}\subset\{\bar\eps_{j_1}=s_{j_1}=-s_{1}\}$, where
$\beps_{-j}$ is the vector obtained from $\beps_{B_1}$ by removing the $j$th entry.
This inclusion yields
\begin{align*}
\bfP(\bar\eps_{j_1}=s_{j_1})&\ge \sum_{j=1}^{\Delta_1}\bfP(\beps_{-j}=\mathbf 1,\eps_j=-1) =
\sum_{j=1}^{\Delta_1}(1-(2\Delta_1)^{-1})^{\Delta_1-1}\times(2\Delta_1)^{-1}\ge \frac14.
\end{align*}
On the other hand, for every $A>0$, we have
\begin{align*}
\bfP\bigg(\sum_{j=1}^{\Delta_1} f_{j}(\bar\eps_ja_j-\bar\eps_{j+1} a_{j+1})> A \bigg)&\le
\frac{1}{A}\,\bfE\Big[\Big|\sum_{j=1}^{\Delta_1}f_{j}(\bar\eps_ja_j-\bar\eps_{j+1}a_{j+1})\Big|\Big].
\end{align*}
We need now to evaluate the expectation of the random variable
$\Upsilon=|\sum_{j=1}^{\Delta_1}f_{j}\bar\eps_j(a_j-\eps_{j+1}a_{j+1})|$. To this end, since all the $a_j$ are nonnegative,
we remark that $\bfE(|a_j-\eps_{j+1}a_{j+1}|) = |a_j-a_{j+1}|+ (a_j\wedge a_{j+1})/\Delta_1$ and, therefore,
$$
\bfE[\Upsilon]\le \sum_{j=1}^{\Delta_1}|f_{j}|\cdot(|a_j-a_{j+1}|+a_j\Delta_1^{-1})\le
\bigg(\sum_{j=1}^{\Delta_1}f_{j}^2\bigg)^{1/2}\bigg(\sum_{j=1}^{\Delta_1}(2|a_j-a_{j+1}|^2+2a^2_j\Delta_1^{-2})\bigg)^{1/2}.
$$
Thus, taking $A=4\big(\sum_{j=1}^{\Delta_1}f_{j}^2\big)^{1/2}\big(\sum_{j=1}^{\Delta_1}(2|a_j-a_{j+1}|^2+2a^2_j\Delta_1^{-2})\big)^{1/2}$,
we get
$$
\bfP\bigg(\sum_{j=1}^{\Delta_1} f_{j}(\bar\eps_ja_j-\bar\eps_{j+1} a_{j+1})> A \bigg)\le 1/4.
$$
Combined with inequality (\ref{prob}), this entails that
$$
\bfP\bigg(\sum_{j=1}^{\Delta_1} f_{j}(\bar\eps_ja_j-\bar\eps_{j+1} a_{j+1})\le A\ \text{and}\ (\bar\eps_1,\bar\eps_{j_1})=(s_1,s_{j_1})\bigg)>0.
$$
Consequently, $\Psi_1\le A$. Applying the same argument to arbitrary $\ell\in[s+1]$, we get
$$
\Psi_\ell\le 4\bigg(\sum_{j\in B_\ell}f_{j}^2\bigg)^{1/2}\bigg(\sum_{j\in B_\ell}(2|a_j-a_{j+1}|^2+2a^2_j\Delta_\ell^{-2})\bigg)^{1/2}.
$$
In view of the Cauchy-Schwarz inequality, we get
\begin{align*}
\|\bu_T\odot \ba_{T}\|_1-\|\bu_{T^c}\odot \ba_{T^c}\|_1
&\le 4\bigg(\sum_{j\in [n]}f_{j}^2\bigg)^{1/2}\bigg(\sum_{\ell=1}^{s+1}\sum_{j\in B_\ell}(2|a_j-a_{j+1}|^2+2a^2_j\Delta_\ell^{-2})\bigg)^{1/2}\\
&\le 4\|\boldf\|_2 \bigg(2\sum_{j\in[n]}|a_j-a_{j+1}|^2+2(s+1)\|\ba\|_\infty^2\max_\ell\Delta_\ell^{-1}\bigg)^{1/2}.
\end{align*}
This completes the proof.
\end{proof}

\begin{proof}[Proof of Proposition~\ref{prop:3.2}]
We apply Theorem~\ref{thm3} with $\gamma=2$. This leads to (\ref{eq:3.1}) with
$$
r_{n,n,T}=
\frac{1+2|T|^{-1}\log(1/\delta)}{\log(p/\delta)}+\frac{4}{\bar\kappa_{T, 2, \bomega}}\le {3}+\frac{4}{\bar\kappa_{T, 2, \bomega}}.
$$
It remains to find a lower bound for $\bar\kappa_{T, 2, \bomega}$. To this end, we resort to Proposition~\ref{prop:3.1} with
$a_j = 1$ for every $j$ belonging to the set $T$ and $a_j = 1-\frac1{2\sqrt{n}}\|(\bfI_n-\Pi_{T})\bx^j\|_2$, $j\in T^c$.
Let $T=\{j_1,\ldots,j_s\}$ and $B_\ell=\llbracket j_{\ell-1},j_\ell\llbracket$ for $\ell\in[s+1]$ with the convention that $j_0=1$ and $j_{s+1}=n+1$.
Since the columns of $\bfX$ are given by $\bx^j=[\mathds 1(i\ge j)]_{i\in[n]}$, the projector
$\Pi_T$ projects onto the subspace of $\mr^n$ containing all the vectors that are constant on the partition $\{B_\ell\}$. Therefore,
one easily checks that $\|(\bfI_n-\Pi_T)\bx^j\|_2 = \sqrt{\frac{(j-j_{\ell-1})({j_\ell}-j)}{j_{\ell}-j_{\ell-1}}}$, for every $j\in B_\ell$.
This implies that $\ba_T=\mathbf 1_T$ and $\|\ba\|_\infty = 1$. Furthermore, we have
\begin{align*}
\sum_{j\in[n]}|a_j-a_{j+1}|^2 &\le \frac1{4n}\sum_{\ell\in[s+1]}\sum_{j\in B_\ell} \frac{\big(\sqrt{(j-j_{\ell-1})({j_\ell}-j)}
-\sqrt{(j+1-j_{\ell-1})({j_\ell}-j-1)}\big)^2}{j_{\ell}-j_{\ell-1}}\\
&\le \frac1{4n}\sum_{\ell\in[s+1]}\sum_{j=1}^{\Delta_\ell} \frac{\big((j-1)(\Delta_\ell-(j-1))
-j(\Delta_\ell-j)\big)^2}{\Delta_{\ell}(\sqrt{(j-1)(\Delta_\ell-(j-1))}+\sqrt{j(\Delta_\ell-j)})^2}\\
&\le \frac1{4n}\sum_{\ell\in[s+1]}\sum_{j=1}^{\Delta_\ell} \frac{\big(2j-3
-\Delta_\ell\big)^2}{\Delta_{\ell}({(j-1)(\Delta_\ell-(j-1))}+{j(\Delta_\ell-j)})}\\
&\le \frac1{4n}\sum_{\ell\in[s+1]}\sum_{j\le \Delta_\ell/2} \frac{\Delta_\ell^2}{\Delta_{\ell}\times j\times\Delta_\ell/2}
=\frac1{2n}\sum_{\ell\in[s+1]}\sum_{j\le \Delta_\ell/2} j^{-1}.
\end{align*}
Since obviously $\Delta_\ell\le n$, we can bound the sum $\sum_{j\le \Delta_\ell/2} j^{-1}$ by $1+\log(n/2)\le 2\log(n)$, provided that
$n\ge 3$. This yields $\sum_{j\in[n]}|a_j-a_{j+1}|^2\le (s+1)\log(n)/n$. Therefore, Proposition~\ref{prop:3.1} implies that
$$
\|\bu_T\|_1-\|\bu_{T^c}\odot \ba_{T^c}\|_1 \le
4\|\bfX\bu\|_2 \bigg(2(s+1)\bigg[\frac{\log(n)}{n}+\frac1{\Delta_{\min}}\bigg]\bigg)^{1/2}.
$$
Using the inequality $s+1\le 2s$, we infer from the above inequality that
$\bar\kappa_{T,2,\bomega}\ge \big\{64(\log(n)+(n/\Delta_{\min}))\big\}^{-1}.$
\end{proof}

\begin{proof}[Proof of Proposition~\ref{prop:4.1}]
Let us denote $\bdelta = \betae-\betaz$. We consider two cases separately.
The first case is when the inequality $\|\betae\|_1\le \frac{\gamma+1}{\gamma-1}\|\betaz\|_1$ is satisfied.
Then, using the Parseval identity, we have
$\ell_n(\betae,\betaz) = \frac1n\|\Pi_T\bfX\bdelta\|_2^2+\frac1n\|(\bfI_n-\Pi_T)\bfX\bdelta\|_2^2$.
The first summand can be bounded using inequality (\ref{eq:2.2}), so we focus on the second summand. Using the triangle inequality, we get
$\|(\bfI_n-\Pi_T)\bfX\bdelta\|_2=\|\sum_{j}(\bfI_n-\Pi_T)\bx^j\delta_j\|_2 \le \sum_{j}\|(\bfI_n-\Pi_T)\bx^j\delta_j\|_2\le
\|\bdelta\|_1\max_{j\in T^c}\|(\bfI_n-\Pi_T)\bx^j\|_2$. Furthermore, we have $\|\bdelta\|_1\le \|\betae\|_1+\|\betaz\|_1
\le \frac{2\gamma}{\gamma-1}\|\betaz\|_1$. Hence, putting these bounds together, we find
\begin{align*}
\ell_n(\betae,\betaz) &\le \frac1n\|\Pi_T\bfX\bdelta\|_2^2+ \frac{4\gamma^2\|\betaz\|_1^2}{n(\gamma-1)^2}
\max_{j\in T^c}\|(\bfI_n-\Pi_T)\bx^j\|_2^2\\
&\le \frac{2\lambda^2|T|}{\nu_T^2}+\frac{4\sigma^*{}^2(|T|+2\log(1/\delta))}{n}+\frac{4\gamma^2\|\betaz\|_1^2}{n(\gamma-1)^2}
\max_{j\in T^c}\|(\bfI_n-\Pi_T)\bx^j\|_2^2,
\end{align*}
with a probability at least $1-\delta$. In the second case, $\|\betae\|_1> \frac{\gamma+1}{\gamma-1}\|\betaz\|_1$, according to
inequalities (\ref{2.8}) and (\ref{2.9}) applied to $\bar\bdelta=\bdelta$, we have
\begin{align*}
\frac2n\|\bfX\bdelta\|_2^2
    &\le \frac2n\bdelta^\top\bfX^\top\bxi+2\lambda(\|\betaz\|_1-\|\betae\|_1)\\
    &\le \frac1{n}\|\bfX\bdelta\|_2^2+\frac{2\sigma^*{}^2(|T|+2\log(1/\delta))}{n} +2\lambda\gamma^{-1}\big(\underbrace{(\gamma+1)\|\betaz\|_1-(\gamma{-1})\|\betae\|_1}_{\le 0}\big),
\end{align*}
with probability at least $1-2\delta$. This completes the proof.
\end{proof}

\begin{proof}[Proof of Proposition~\ref{prop:4.2}]
Applying Theorem~\ref{thm:4} to $\lambda=\sigma^*(\log(n/\delta)/(nk))^{1/2}$, in conjunction with the bound $\rho_T\le (2k)^{-1/2}$,
we get that the inequality
\begin{align}\label{eq:4.5}
\frac1n\|\hat\boldf^{\rm TV}-\boldf^*\|_2^2
&\le \inf_{\bar\boldf\in\mr^n} \bigg\{ \frac1n\|\bar\boldf-\boldf^*\|_2^2  +
4 \lambda \|\bar\boldf\|_{\rm TV}\bigg\} + \frac{2\sigma^*{}^2(k+2\log(1/\delta))}{n}
\end{align}
holds true with a probability at least $1-2\delta$. Replacing $\bar\boldf$ by $\boldf^\uparrow$, and replacing $\lambda$ by its expression,
we get
\begin{align*}%\label{eq:4.6}
\frac1n\|\hat\boldf^{\rm TV}-\boldf^*\|_2^2
&\le \frac1n\|\boldf^\uparrow-\boldf^*\|_2^2  + \frac{4\sigma^*{}^2\log(1/\delta)}{n}+
\underbrace{4\sigma^*\Big(\frac{\log(n/\delta)}{nk}\Big)^{1/2}\|\boldf^\uparrow\|_{\rm TV}+ \frac{2\sigma^*{}^2k}{n}}_{:=\Psi(k)},
\end{align*}
with probability $\ge 1-2\delta$. One readily checks that $x\mapsto \Psi(x)$ achieves its (global) minimum at
$x_{\min}=(\|\boldf^\uparrow\|_{\rm TV}/\sigma^*)^{2/3}(n\log(n/\delta))^{1/3}$. Furthermore, the definition of $k$ entails that
$1/k\le 1/x_{\min}$ and $k\le x_{\min}+1$. This yields $\Psi(k)\le \Psi(x_{\min})+2\sigma^*{}^2/n$ and the desired result follows.
\end{proof}
\begin{proof}[Proof of Proposition~\ref{prop:4.3}]
We start by observing that $\mathcal H^n_{\alpha,L}$ is a closed convex subset of $\mr^n$. Therefore, for every $\bar\boldf\in\mathcal
H^n_{\alpha,L}$ and for $\boldg^*=\argmin_{\boldf\in\mathcal H^n_{\alpha,L}}\|\boldf-\boldf^*\|_2$, we have
\begin{align}\label{eq:4.71}
\|\bar\boldf-\boldf^*\|_2^2 \le \|\boldg^*-\boldf^*\|_2^2 +\|\bar\boldf-\boldg^*\|_2^2.
\end{align}
Now, let $T=\{a_1,a_2,\ldots,a_k\}\subset [n]$ be any set satisfying $a_1=1$ and
$0\le a_{j+1}-a_j\le 2n/k$, $\forall j\in[k]$. This set induces
the partition $\mathcal T = \{I_1,\ldots,I_{k}\}$ where each $I_j=\llbracket a_j,a_{j+1}\llbracket$ is an interval of
length smaller or equal to $2n/k$. We define $\bar\boldf$ as an approximation of $\boldg^*$ by
a piecewise linear vector on the partition $\mathcal T$: $\bar f_i = g^*_{a_j}+\frac{i-a_j}{a_{j+1}-a_j}
\,\{g^*_{a_{j+1}}-g^*_{a_j}\}$, $\forall i\in I_j$. On the one hand, one can easily check that
$\bar\boldf$ belongs to the set $\mathcal H^n_{\alpha,L}$ and, therefore, satisfies inequality (\ref{eq:4.71}). In conjunction with
(\ref{eq:4.5}), this implies that for $\lambda=\sigma^*(\log(n/\delta)/(nk))^{1/2}$, the inequalities
\begin{align*}%\label{eq:4.8}
\frac1n\|\hat\boldf^{\rm TV}-\boldf^*\|_2^2
&\le \inf_{\bar\boldf\in\mr^n} \bigg\{ \frac1n\|\bar\boldf-\boldf^*\|_2^2  +
4 \lambda \|\bar\boldf\|_{\rm TV}\bigg\} + \frac{2\sigma^*{}^2(k+2\log(1/\delta))}{n}\nonumber\\
&\le \frac1n\|\boldg^*-\boldf^*\|_2^2  +\frac1n\|\bar\boldf-\boldg^*\|_2^2+
4 \lambda \|\bar\boldf\|_{\rm TV} + \frac{2\sigma^*{}^2(k+2\log(1/\delta))}{n}
\end{align*}
hold true with a probability at least $1-2\delta$. Since $\boldg^*\in \mathcal H^n_{\alpha,L}$, we have
\begin{align*}
\|\bar\boldf-\boldg^*\|_2^2&\le \sum_{j}\sum_{i\in I_j}|g^*_{a_j}+\frac{i-a_j}{a_{j+1}-a_j}
\,\{g^*_{a_{j+1}}-g^*_{a_j}\}-g^*_{i}|^2\\
&\le
\sum_{j}\sum_{i\in I_j}(\frac{a_{j+1}-i}{a_{j+1}-a_j}|g^*_{a_j}-g^*_{i}|^2+\frac{i-a_j}{a_{j+1}-a_j}
\,|g^*_{a_{j+1}}-g^*_{i}|^2)\\
&\le n L^2n^{-2\alpha}(2n/k)^{2\alpha}\le 4nL^2k^{-2\alpha}.
\end{align*}
On the other hand, since $\bar\boldf$ is piecewise constant, it holds that
$\|\bar\boldf\|_{\rm TV} = \sum_j |g^*_{a_{j+1}}-g^*_{a_j}|\le k L n^{-\alpha}(2n/k)^\alpha$ $\le 2L k^{1-\alpha}$.
Combining all these bounds, we get that with probability at least $1-2\delta$,
\begin{align*}
\frac1n\|\hat\boldf^{\rm TV}-\boldf^*\|_2^2
&\le \frac1n\|\boldg^*-\boldf^*\|_2^2  + \frac{4\sigma^*{}^2\log(1/\delta)}{n}\nonumber\\
&\qquad\qquad+
4L^2k^{-2\alpha}+
8L\sigma^*\Big(\frac{\log(n/\delta)}{nk}\Big)^{1/2}k^{1-\alpha}  + \frac{2\sigma^*{}^2k}{n}.
\end{align*}
The inequality between the geometric and arithmetic means yields $8L\sigma^*\Big(\frac{\log(n/\delta)}{nk}\Big)^{1/2}k^{1-\alpha} \le
8L^2k^{-2\alpha}+2\sigma^*{}^2\frac{k\log(n/\delta)}{n}$. Thus, with probability at least $1-2\delta$, we have
\begin{align*}%\label{eq:4.10}
\frac1n\|\hat\boldf^{\rm TV}-\boldf^*\|_2^2
&\le \frac1n\|\boldg^*-\boldf^*\|_2^2  + \frac{4\sigma^*{}^2\log(1/\delta)}{n}
+12L^2k^{-2\alpha}+ \frac{4\sigma^*{}^2k\log(n/\delta)}{n}.
\end{align*}
Using the inequalities $k-1\le (L^2n/(\sigma^*{}^2\log(n/\delta)))^{1/(2\alpha+1)}\le k$ we get
\begin{align*}%\label{eq:4.10}
12L^2k^{-2\alpha}&\le
12L^2\bigg(\frac{L^2n}{\sigma^*{}^2\log(n/\delta)}\bigg)^{-2\alpha/(2\alpha+1)}  =
12L^2\bigg(\frac{\sigma^*{}^2\log(n/\delta)}{L^2n}\bigg)^{2\alpha/(2\alpha+1)},
\end{align*}
and
\begin{align*}%\label{eq:4.10}
\frac{4\sigma^*{}^2k\log(n/\delta)}{n}&\le
\frac{4\sigma^*{}^2\log(n/\delta)}{n}\times \bigg(\frac{L^2n}{\sigma^*{}^2\log(n/\delta)}\bigg)^{1/(2\alpha+1)}+
\frac{4\sigma^*{}^2\log(n/\delta)}{n}\\
&=
4L^2\bigg(\frac{\sigma^*{}^2\log(n/\delta)}{L^2n}\bigg)^{2\alpha/(2\alpha+1)}+
\frac{4\sigma^*{}^2\log(n/\delta)}{n}.
\end{align*}
To complete, we use the fact that $\|\boldg^*-\boldf^*\|_2^2=\inf_{\boldf\in\mathcal H^n_{\alpha,L}}
\frac1n\|\boldf-\boldf^*\|_2^2$.
\end{proof}

\begin{proof}[Proof of Proposition~\ref{prop:3}]
Let $\bu^*$ be any vector satisfying the cone constraint $\|\bu_{T^c}^*\|_1\le \bar c\|\bu_T^*\|_1$. Clearly, the vector $-\bu^*$
satisfies the same constraint. For any $J\subset T^c$, we define the random vector $\bzeta = \sigma^*{}^{-1}(\bfX_J^\top\bfX_J)^\dag\bfX_J^\top\bxi$.
Let $\bv \in\mr^p$ be the random vector defined by $\bv_{J^c}=0$ and $\bv_J = \alpha \bzeta$, where  $\alpha = (\bar c\|\bu_T^*\|_1-\|\bu^*_{T^c}\|_1)/\|\bzeta\|_1$ is a positive number ensuring that the vectors $\pm\bu^*+\bv$ satisfy the cone constraint.
Indeed, the triangle inequality implies that
\begin{align*}
\|(\pm\bu^*+\bv)_{T^c}\|_1- \bar c\|(\pm\bu^*+\bv)_T\|_1
		&=\|(\pm\bu^*+\bv)_{T^c}\|_1- \bar c\|\bu^*_T\|_1\\
		&\le \|\bu^*_{T^c}\|_1+\|\bv_{J}\|_1- \bar c\|\bu^*_T\|_1=0.
\end{align*}
Further, we remark that $\bxi^\top\bfX\bv= (\alpha/\sigma^*)\bxi^\top\Pi_J\bxi \ge 0$ and therefore
\begin{align*}
\bar\eta&\ge \max_{\bu\in\{\pm\bu^*\}}\frac{|\bxi^\top\bfX(\bu+\bv)|}{\sigma^*\|\bfX(\bu+\bv)\|_2}\ge \max_{\bu\in\{\pm\bu^*\}}\frac{|\bxi^\top\bfX(\bu+\bv)|}{\sigma^*(\|\bfX\bu^*\|_2+\|\bfX\bv\|_2)}\\
&= \frac{|\bxi^\top\bfX\bu|+\bxi^\top\bfX\bv}{\sigma^*(\|\bfX\bu^*\|_2+\|\bfX\bv\|_2)}\ge {\frac{\bxi^\top\bfX\bv}{\sigma^*(\|\bfX\bu^*\|_2+\|\bfX\bv\|_2)}}{:=\eta_1}.
\end{align*}
Using the definition of $\bv$, we get the relations
$\bxi^\top\bfX\bv= \bxi^\top\bfX_J\bv_J=(\alpha/\sigma^*)\|\Pi_J\bxi\|_2^2$ and
$\|\bfX\bv\|_2= \|\bfX_J\bv_J\|_2 = (\alpha/\sigma^*)\|\Pi_J\bxi\|_2$ implying that
$$
\eta_1=\frac{\alpha\|\Pi_J(\bxi/\sigma^*)\|_2^2}{\|\bfX\bu^*\|_2+\alpha\|\Pi_J(\bxi/\sigma^*)\|_2}.
$$
To complete the proof, it remains to find appropriate lower bounds for the terms $\|\Pi_J(\bxi/\sigma^*)\|_2^2$ and $\alpha$.
Since $\|\Pi_J(\bxi/\sigma^*)\|_2^2\sim \chi^2_{|J|}$, we have
$\bfP\big\{\|\Pi_J\bxi\|_2^2\ge  \sigma^*{}^2\big(|J|-2\sqrt{|J|\log(2/\delta')}\big)\big\}\ge 1-(\delta'/2)$. Therefore, for every
$|J|\ge 10> (8/3)^2\log 4$, we have $\bfP\big\{\|\Pi_J\bxi\|_2^2\ge  \sigma^*{}^2|J|/4\big\}\ge 1-1/4$. Let  $\bfX_J =\bfV\bLambda\bfU^\top$
be the singular value decomposition of $\bfX_J$ and $\tilde\bxi = \bfV^\top\bxi/\sigma^*\in\mr^{|J|}$ with Gaussian distribution
$\mathcal N(0,\bfI_{|J|})$. We remark that
\begin{align*}
\|(\bfX_J^\top\bfX_J)^\dag\bfX_J^\top\bxi \|_1 & = \sigma^*\|\bfU\bLambda^{-1}\tilde\bxi \|_1\\
&\le \sigma^*|J|^{1/2}\|\bfU\bLambda^{-1}\tilde\bxi \|_2\\
&= \sigma^*|J|^{1/2}\|\bLambda^{-1}\tilde\bxi \|_2\\
&\le \sigma^*(|J|/n\lambda^2_{\min,J})^{1/2}\|\tilde\bxi \|_2
\end{align*}
Therefore, using tail bounds for the $\chi^2$ distribution, we find
\begin{align*}
\bfP\Big(\|\bzeta \|_1  \le (|J|/n\lambda^2_{\min,J})^{1/2}\big(
\underbrace{|J|^{1/2}+\sqrt{2\log(4)}}_{\le 2{|J|^{1/2}}}\big)\Big)\ge 1-\frac14.
\end{align*}
From this inequality we infer that $\alpha\ge n^{1/2}\lambda_{\min,J}(\bar c\|\bu^*_T\|_1-\|\bu^*_{T^c}\|_1) /(2|J|)$
with probability at least $3/4$. Combining the lower bounds obtained for $\alpha$ and $\|\Pi_J\bxi\|_2^2$, we get
$$
\bfP\Bigg(\eta_1\ge \frac{|J|}{4}\bigg\{{\frac{2|J| \cdot\|\bfX\bu^*\|_2}{n^{1/2}\lambda_{\min,J}\bar c(\|\bu^*_T\|_1-\bar c^{-1}\|\bu^*_{T^c}\|_1)}
+ \frac{|J|^{1/2}}{2}} \bigg\}^{-1}\Bigg)\ge 1/2.
$$
This implies that
$$
\bar\varrho_{T,\bar c} \ge \bigg\{\frac{8\|\bfX\bu^*\|_2}{n^{1/2}\lambda_{\min,J}(\|\bu^*_T\|_1-\bar c^{-1}\|\bu^*_{T^c}\|_1)}
+ 2|J|^{-1/2}\bigg\}^{-1}.
$$
Since this is true for every $\bu^*$ belonging to the cone, we can take the supremum of the right hand-side to get
$$
\bar\varrho_{T,\bar c} \ge \frac{1}{8\lambda_{\min,J}^{-1}(\kappa_{T,\bar c}/|T|)^{1/2}
+ 2|J|^{-1/2}},
$$
and the desired result follows.
\end{proof}

\section*{Acknowledgments}
Johannes Lederer acknowledges partial financial support as fellow of the Swiss National Science Foundation.
The work of Arnak Dalalyan was partially supported by the grant Investissements d'Avenir (ANR-11-IDEX-0003/Labex Ecodec/ANR-11-LABX-0047).

\appendix

\section{Computation of the compatibility factors by sequential convex programming} The contribution of the present paper is theoretical. However it may be of practical interest to evaluate the risk bounds presented in our main results in order to understand how accurate the Lasso
prediction is. To this end, one may often need to compute, at least approximately, the compatibility factor $\kappa_{T,\bar c}$ or its weighted
counterpart $\bar\kappa_{T,\gamma,\bomega}$. In general, this task is difficult to accomplish since the computation of the aforementioned
factors amounts to minimizing a nonconvex function over a nonconvex subset of $\mr^p$ with large dimensionality $p$. (For a different
procedure, efficiently verifiable conditions entailing theoretical guarantees for sparse recovery are proposed in \cite{JudNem11b}.)
There is, however, a particular case where this task may be solved with a reasonable computational complexity.
This corresponds to subsets $T$ of small or
moderately large cardinalities. In fact, in the remaining of this subsection we will show that for every $T\subset [p]$ and every
$\bar c,\epsilon>0$, one can find an interval $[\kappa_*,\kappa^*]$ of length at most $\epsilon$ and containing
$\bar\kappa_{T,\bar c}$ by solving at most $2^{|T|}\log_2(|T|/\epsilon)$ convex programs. Furthermore, each of these convex
programs is a second-order cone program (SOCP) and the global computation can be split into $2^{|T|}$ parallel programs, each program involving
$\log_2(|T|/\epsilon)$ SOCPs. Note that an alternative approach proposed in \cite{Gautier} is to lower-bound the compatibility
factor by a quantity that can be efficiently computed by linear programming.

The procedure we propose relies on the well-known bisection algorithm. Since we know that $\bar\kappa_{T,\bar c}$ always belong to
the interval $[0,|T|]$, we start by setting $\kappa_*=0$ and $\kappa^*=|T|$. Then, at each step of iteration, we set
$\kappa=(\kappa_*+\kappa^*)/2$ and $\mu=(n\kappa/|T|)^{1/2} \bar c^{-1}$ and solve the problem
\begin{align}\label{quasiconv1}
\begin{matrix}
\displaystyle\text{minimize} \quad \|\bfX \bdelta\|_2 +\mu\|\bdelta_{T^c}\|_1\\[6pt]
\text{subject to}\quad \|\bdelta_T\|_1=1.
\end{matrix}
\end{align}
If the solution $\bdelta^\mu$ of this problem satisfies $\|\bfX \bdelta^\mu\|_2 +\mu\|\bdelta^\mu_{T^c}\|_1 \le  \mu \bar c$, then
we leave $\kappa_*$ unchanged and decrease  $\kappa^*$ by setting
$$
\kappa^*=\frac{|T|\cdot\|\bfX\bdelta^\mu\|_2^2}{n(1-\bar c^{-1}\|\bdelta^\mu_{T^c}\|_1)^2}.
$$
(Note that the right hand-side is always less than or equal to $\kappa$). If, in contrast, the solution $\bdelta^\mu$ of
problem (\ref{quasiconv1}) satisfies $\|\bfX \bdelta^\mu\|_2 +\mu\|\bdelta^\mu_{T^c}\|_1 >  \mu \bar c$,
then we leave $\kappa^*$ unchanged and increase  $\kappa_*$ by setting $\kappa_*=\kappa$. We iterate this process until we get
$\kappa^*-\kappa_*\le \epsilon$. Since at the first step the gap $\kappa^*-\kappa_*$ equals $|T|$ and at each step this gap
is divided by at least a factor $2$, the total number of iterations to get precision $\epsilon$ is not larger than $\log_2(|T|/\epsilon)$.

\begin{algorithm}[t]
\caption{Pseudo-code for computing the compatibility factor}
\label{algo1}
\begin{algorithmic}[1]
\Require $n\times p$ matrix $\bfX$, set $T\subset [p]$, constant $\bar c>0$, precision $\epsilon>0$
\Ensure interval $[\kappa_*,\kappa^*]$ containing the compatibility factor $\kappa_{T,\bar c}$
\State $\kappa_*\gets 0$ and $\kappa^*\gets |T|$
\While {$\kappa^*-\kappa_*<\epsilon$}
\State $\kappa\gets (\kappa^*+\kappa_*)/2$
\State $\mu\gets (n\kappa/|T|)^{1/2}\bar c^{-1}$
\For {$\bs\in\{\pm1\}^{|T|}$}
\State $v^{\bs} \gets \min \{\|\bfX \bdelta\|_2 +\mu\|\bdelta_{T^c}\|_1\}$  subject to $\bs^\top\bdelta_T=1$ and $s_j\delta_j\ge 0$, $\forall j\in T$
\State $\bdelta^{\bs} \gets \text{arg}\min \{\|\bfX \bdelta\|_2 +\mu\|\bdelta_{T^c}\|_1\}$  subject to $\bs^\top\bdelta_T=1$ and $s_j\delta_j\ge 0$, $\forall j\in T$
\EndFor
\If {$\min_{\bs} v^{\bs}> \mu\bar c$}
\State $\kappa_*\gets\kappa$
\Else
\State $\bdelta^\mu\gets \text{arg}\min_{\bdelta\in \{\bdelta^{\bs}:\bs\in \{\pm 1\}^{|T|}\} } \big(\|\bfX \bdelta\|_2
+\mu\|\bdelta_{T^c}\|_1\big)$
\State $\kappa^*\gets |T|\cdot\|\bfX\bdelta^\mu\|_2^2/(n\{1-\bar c^{-1}\|\bdelta^\mu_{T^c}\|_1\}^2)$
\EndIf
\EndWhile
\end{algorithmic}
\end{algorithm}

Let us now analyze the complexity of the optimization problem (\ref{quasiconv1}). The objective function of this problem is convex but
the set of feasible solutions is not. Interestingly, for every $\bs\in\{\pm1\}^{|T|}$, if we restrict the optimization to the orthant
$\{\bdelta: s_j \delta_j\ge 0,\ \forall j\in T\}$, the constraints become convex as well. Thus, our proposal consists in replacing
(\ref{quasiconv1}) by $2^{|T|}$ optimization problems
\begin{align}\label{quasiconv2}
\begin{matrix}
\displaystyle\text{minimize} \quad \|\bfX \bdelta\|_2 +\mu\|\bdelta_{T^c}\|_1\\[6pt]
\text{subject to}\quad \bs^\top\bdelta_T=1,\ \text{ and }\ s_j\delta_j\ge 0,\ \forall j\in T.
\end{matrix}
\end{align}
For every $\bs\in\{\pm 1\}^{|T|}$, (\ref{quasiconv2}) can be rewritten as a SOCP using standard arguments.
Denoting by $\bdelta^{\mu,\bs}$ any solution of (\ref{quasiconv2}), we determine a solution of (\ref{quasiconv1})
by minimizing the common objective function of the above optimization problems over the finite set
$\{\bdelta^{\mu,\bs}:\bs\in \{\pm 1\}^{|T|}\}$, that is
$$
\bdelta^\mu = \text{arg}\min_{\bdelta\in \{\bdelta^{\mu,\bs}:\bs\in \{\pm 1\}^{|T|}\} } \big(\|\bfX \bdelta\|_2 +\mu\|\bdelta_{T^c}\|_1\big).
$$
lead to the procedure summarized in Algorithm~\ref{algo1}.
% \pagebreak

\bibliographystyle{alpha}      % basic style, author-year citations
\bibliography{Literature}

\end{document}